\documentclass[11pt,reqno]{amsart}
\usepackage{graphicx, subfigure}
\usepackage{amssymb}
\usepackage{epstopdf}
\usepackage{amssymb}
\usepackage[a4paper, total={6.6in, 9.9in}]{geometry}
\usepackage{multirow}
\usepackage{bm}
\usepackage{amsmath,amsfonts,amsthm,mathrsfs,amssymb}
\usepackage{mathtools}
\usepackage{framed}
\usepackage{exscale}
\usepackage{relsize}
\usepackage{enumerate}
\usepackage{epstopdf}
\usepackage{latexsym}
\usepackage[usenames]{color}
\usepackage{graphics}
\usepackage{framed}
\usepackage{tikz}
\usepackage{enumerate}
\usepackage{xcolor}

\newtheorem{definition}{Definition}[section]
\newtheorem{theorem}{Theorem}[section]
\newtheorem{lemma}{Lemma}[section]

\theoremstyle{definition}

\theoremstyle{remark}
\newtheorem{remark}[theorem]{Remark}

\numberwithin{equation}{section}

\newcounter{saveeqn}

\begin{document}

\title[Inverse spectral problem with low regularity refractive index]{Inverse spectral problem with low regularity refractive index}

\author{Kewen Bu}
\address{Department of Mathematics, Jilin University, Changchun, Jilin, China.}\email{bukw22@mails.jlu.edu.cn}
\address{Department of Mathematics, Simon Fraser University, Burnaby, BC, Canada}\email{kba136@sfu.ca}

\author{Youjun Deng}
\address{School of Mathematics and Statistics, Central South University, Changsha, Hunan, China.}
\email{youjundeng@csu.edu.cn}

\author{Yan Jiang}
\address{Department of Mathematics, City University of Hong Kong, Hong Kong SAR, China.}
\email{yjian24@cityu.edu.hk}

\author{Kai Zhang}
\address{Department of Mathematics, Jilin University, Changchun, Jilin, China.}
\email{zhangkaimath@jlu.edu.cn}


\thanks{}

\date{}


\maketitle

\begin{abstract}

This article investigates the unique determination of a radial refractive index $n$ from spectral data. First, we demonstrate that for piecewise twice continuously differentiable functions, $n$ is not uniquely determined by the special transmission eigenvalues associated with radially symmetric eigenfunctions. Subsequently we prove that if $n \in \mathcal{M}$ is twice continuously differentiable functions (or continuously differentiable functions with Lipschitz continuous derivative), then $n$ is uniquely determined on $[0,1]$ by all special transmission eigenvalues when supplemented by partial a priori information on the refractive index.

\end{abstract}


\section{Introduction}\label{sec:Introduction}

The analysis of interior transmission problems has been a subject of considerable interest since their seminal introduction by Colton and Monk (\cite{CM88}). In the ensuing years, these problems have stimulated a rich body of literature dedicated to both their direct analysis—concerning existence, uniqueness, and spectral theory—and their inverse counterparts, which focus on the reconstruction of medium properties (\cite{CCH16,CK19}). Let $ k $ be the wave number,~$k\in\mathbb{C}$ and $ n(x) $ the refractive index, with $ n(x) = 1 $ for $ x \in \mathbb{R}^d \setminus \overline{D} $, where $ d = 2, 3 $, and $ D $ be a domain bounded, simply connected with a piecewise smooth boundary $\partial D$. The associated medium scattering problem is formulated as:
\begin{equation}\label{eq:medium scattering}
	\left\{
	\begin{aligned}
		\Delta u + k^2 n(x) u = 0, \quad x~\text{in} \quad \mathbb{R}^d, \\
		u(x) = \exp(ikx \cdot l) + u^s(x), \\
		\lim_{r \to \infty} r \left( \frac{\partial u^s}{\partial r} - ik u^s \right) = 0,
	\end{aligned}
	\right.
\end{equation}
where $r=|x|$ and $l$ is a unit vector. The inverse medium problem associated with \eqref{eq:medium scattering} is to reconstruct the refractive index $n$ from the far-field pattern of the scattered wave $u^s$. Sampling methods can be used to study the inverse medium problem; however, these methods can only be applied when the kernel of the far-field operator is empty (\cite{CK19}). If the incident pattern lies within the kernel of the far-field operator, the corresponding incident wave does not generate a reflected or scattered wave (\cite{CPS07}). In such cases, there will be no far-field pattern to reconstruct the refractive index.

The interior transmission problem provides a powerful tool for investigating the kernel of the far-field operator. The effectiveness of the linear sampling method for the case when $|\text{Re}(A)|<1$,~where $A$ is a $3\time3$ matrix valued function,~the real part of which describes the physical properties of the medium was established through the interior transmission problem in the foundational work \cite{CH03}. Building on this, Cakoni and collaborators further developed the theory across several key areas, providing a quantitative analysis of how small inclusions perturb eigenvalues \cite{CMR15} and deriving a rigorous electromagnetic variational formulation \cite{CH07}. The interior transmission problem corresponds to the medium scattering problem \eqref{eq:medium scattering} is:
\begin{equation}\label{eq:transmission eigenvalue problem}
	\left\{
	\begin{array}{rlllr}
		\Delta  u + k^2 n(x) u & = 0, &x~ \text{in } D, \\
		\Delta  v + k^2 v & = 0, &x~ \text{in } D, \\
		u & = v, & \text{on } \partial D, \\
		\frac{\partial u}{\partial \nu} & = \frac{\partial v}{\partial \nu}, & \text{on } \partial D, \\
	\end{array}
	\right.
\end{equation}
where $D = B(0, 1)$ is the unit ball and $n$ is a radial function with $n \in L^2[0, 1]$ (cf.\cite{CCH16}). The transmission eigenvalues of the interior transmission problem, which correspond to values of $k$ for which there exist nontrivial solutions $u$ and $v$ to the problem \eqref{eq:transmission eigenvalue problem}, are crucial to investigating the kernel of the far-field operator. These eigenvalues are not only important for studying the medium scattering problem and the inverse medium problem, but also help in solving the interior transmission problem itself. For example, the transmission eigenvalues can be determined from the far-field pattern of the scattered wave $u^s$ in \eqref{eq:medium scattering}, and these eigenvalues can then be used to estimate the refractive index $n(x)$ (\cite{CCH09,CCH10,CPS07}).

During the past decade, both direct and inverse problems associated with interior transmission have received significant attention. The direct problem focuses on the existence and distribution of transmission eigenvalues for a given refractive index $n$. Considerable progress has been made in solving this problem. For spherically stratified media, Colton proved the existence of an infinite discrete set of transmission eigenvalues, each associated with spherically symmetric eigenfunctions.~ It has also been shown that the set of transmission eigenvalues is at most discrete when the domain $D$ is a general bounded region (\cite{CKP89,RS91}). Furthermore, for a general connected region $D$ with a real-valued refractive index function $n$, Sylvester proved the existence of transmission eigenvalues and derived Faber-Krahn-type upper and lower bounds for the first transmission eigenvalue (\cite{PS08}). In \cite{CGH10}, Cakoni established the existence of countably infinite real transmission eigenvalues for isotropic and anisotropic media, addressing both the Helmholtz and Maxwell equations under more general conditions. For general regions $D$, these real eigenvalues are countable and accumulate only at infinity (\cite{CCH16}). This existence result extends to cases where an impenetrable obstacle $D_0$ is embedded within $D$ (\cite{CCH12}). In contrast, results on complex transmission eigenvalues are limited to specific geometries, such as circular or spherical domains (\cite{CL13,CL17,CLM15,LC12}). Alongside these theoretical advances, numerical methods for computing these eigenvalues have also been developed (\cite{AS15, CMS14, GJSX16, LHLL15}).

The inverse interior transmission eigenvalue problem---classified as an inverse spectral problem---concerns the reconstruction of the refractive index $n$ from a given set of transmission eigenvalues, which may be either partial or complete. The majority of established results for this problem pertain to the case where the domain $D$ is a circle in $\mathbb{R}^2$ or a sphere in $\mathbb{R}^3$. Initial research in this area focused on recovering $n$ using the complete spectrum of transmission eigenvalues. The foundational work on this inverse problem was conducted by Mclaughlin and Polyakov (\cite{MP94}), which serves as an inspiration for the present study. Their results were later expanded with additional theoretical and numerical findings by Mclaughlin, Polyakov, and Sacks (\cite{MPS94}). Subsequently, Cakoni proved that the refractive index $n$ is uniquely determined by the entire set of transmission eigenvalues when $D$ is the unit ball in $\mathbb{R}^3$ (\cite{CCG10}); this work also provided theoretical and numerical characterizations of eigenvalue-free zones. Furthermore, Gintides extended this uniqueness result to a broader class of media, demonstrating that it still holds for piecewise twice continuously differentiable radial functions with jump discontinuities, provided the function has well-defined left and right limits and derivatives at the points of discontinuity (\cite{GP17}).

In the context of the inverse spectral problem, partial transmission eigenvalue data have also been employed for the determination of the refractive index $n$. For instance, Cakoni established that in the case of a constant refractive index, the value of $n$ is uniquely determined by the smallest positive transmission eigenvalue alone (\cite{CCG10}). However, for a general (non-constant) refractive index $n$, the smallest positive eigenvalue can only be used to estimate its maximum value (\cite{CPS07}). Concerning the unique determination of a general function $n$, this was first achieved using transmission eigenvalues associated with radially symmetric eigenfunctions, a result established in \cite{AGP11}. Colton later provided a simplified proof of this uniqueness result in \cite{CL13}.

For inverse problems, a central objective is to achieve unique determination of the unknown with minimal information and under the least restrictive constraints. In the context of inverse spectral problems, existing results on determining the refractive index $n$ from partial transmission eigenvalues typically impose strong regularity conditions, such as requiring $n\in C^2[0,1]$ (\cite{CL13}) or $n\in C^1[0,1]$ with a square-integrable second derivative (\cite{AGP11}). The work of Gintides represents a significant relaxation of these smoothness requirements (\cite{GP17}). It considers a two-layer medium, reducing the assumption on $n$ to a two-segment piecewise continuously differentiable functions with prescribed jump conditions at the interface. Nevertheless, this approach is specifically tailored to a medium with exactly two layers, and extending it to a more general multi-layer or arbitrary piecewise smooth structure appears challenging. Furthermore, the uniqueness result in \cite{GP17} relies on knowledge of the entire spectrum of transmission eigenvalues. Motivated by these limitations, we seek to generalize these findings in two key directions: firstly, to arbitrary piecewise twice continuously differentiable ($C_p^2[0,1]$) refractive indices and to those that are continuously differentiable function with a Lipschitz continuous derivative($C^{1,1}[0,1]$); and secondly, to establish uniqueness using only a subset of special transmission eigenvalues coupled with partial a priori information on the refractive index.

This paper aims to relax the regularity requirements imposed on the refractive index $n$ in existing theory. As a novel contribution, we construct a counterexample demonstrating that unique determination of $n$ is impossible using only special transmission eigenvalues; this establishes a fundamental lower bound on the informational data necessary for the inverse problem. Conversely, we prove that supplementing these eigenvalues with partial information on $n$ does guarantee uniqueness. This finding suggests that the role of stringent regularity conditions can be functionally substituted by the a priori knowledge of $n$ on a subinterval, thereby establishing a corresponding upper bound on the required information. A central long-term objective is to narrow the gap between these lower and upper bounds, producing increasingly precise estimates that converge towards the minimal information needed for unique reconstruction. The results herein complement and significantly extend the foundational work found in~\cite{AGP11, CL13, GP17}. Moreover, based on Diophantine approximation, we only required $0<n_*<n(r)<n^{*}$ for $r\in [0,1]$, which is totally different from existing work either $ 0<n<1$ or $n>1$.

The remainder of this paper is structured as follows. In Section 2, we present the necessary preliminary material alongside our main results. Section 3 is devoted to establishing the limits and potential of unique determination. We begin by constructing a counterexample which proves that for a piecewise twice continuously differentiable index $n$, special transmission eigenvalues alone are insufficient for unique determination.  In Section 4, we prove Theorem~\ref{thm:uniqueness}, our central positive result, which asserts that for $n \in C_p^2[0,1] \cap \mathcal{M}$,~$\mathcal{M}$ a function space of piecewise continuous functions in which the difference between any two functions changes sign only finitely many times,~uniqueness is achieved when these special eigenvalues are supplemented with a priori information on $n$. In Section 5, we prove Theorem~\ref{thm:uniqueness_C1}, extending this uniqueness result to the case where $n\in C^{1,1}[0,1] \cap \mathcal{M}$, again using a priori information on $n$. The final section presents the conclusions and outlines potential directions for future work. 


\section{Mathematical setup and summary of major result}\label{sec:major result}

In this section, we provide the mathematical framework for both forward and inverse spectral problems and state the main results.
First of all, we present some fundamental definitions.

 \begin{definition}\label{def:transmission eigenvalue}
	 	A number $ k $ is called a transmission eigenvalue if the equations \eqref{eq:transmission eigenvalue problem} have nontrivial solutions $ u, v \in L^2(D) $ such that $ u - v \in H_0^2(D) $. Here, $ u $,~$ v $ are the eigenfunctions corresponding to the eigenvalue $ k $.
 \end{definition}

 \begin{definition}\label{def:special transmission eigenvalue}
	 A number \( k \) is termed a \emph{special transmission eigenvalue} if it is a transmission eigenvalue corresponding to a radially symmetric eigenfunction for the equations in \eqref{eq:transmission eigenvalue problem}.
	 \end{definition}

\begin{definition}\label{def:space M}
The space $\mathcal{M}$ is defined as the set of all piecewise continuous functions on $[0,1]$ satisfying the following property: for any $f_1, f_2 \in \mathcal{M}$, there exists a partition
\begin{equation*}
0 = x_0 < x_1 < \dots < x_{S-1} < x_S = 1
\end{equation*}
with $S \in \mathbb{N}^+$ such that for each $s = 0, 1, \dots, S-1$,
\begin{equation*}
H\bigl(f_1(x) - f_2(x)\bigr) = H\bigl(f_1(y) - f_2(y)\bigr)
\quad \text{for all } x, y \in (x_s, x_{s+1}).
\end{equation*}
Here $H$ denotes the Heaviside step function.
\end{definition}

 \begin{definition}\label{def:rational linear independent}
	    For $ \{\alpha_j \}_{j=1}^{n}\subset  \mathbb{R}$, if there exist rational numbers $ \{ r_j \}_{j=1}^{n} \subset \mathbb{Q}$, not all zero, such that $\sum_{j=1}^{n} r_j \alpha_j = 0 $, then $\{\alpha_j \}_{j=1}^{n}$ are linearly dependent over $ \mathbb{Q} $. Otherwise, they are linearly independent over $ \mathbb{Q} $.
\end{definition}

Since we employ special transmission eigenvalues, the radially symmetric solutions can be expressed in the form
\begin{equation}\label{eq:radial solutions}
	u = \frac{y}{r}, \quad  v = \frac{y_0}{r},
\end{equation}
and by substituting these expressions into equation \eqref{eq:transmission eigenvalue problem}, we obtain
\begin{equation}\label{eq:original_ode}
	\left\{
	\begin{aligned}
		&y'' + k^2 n(r) y = 0, \quad r \in (0, 1), \\
		&y_0'' + k^2 y_0 = 0, \quad r \in (0, 1), \\
		&y(1) = y_0(1), \\
		&y'(1) = y_0'(1). \\
	\end{aligned}
	\right.
\end{equation}
Here, we assume that $n$ is $C_p^2[0,1]$ (or $C^{1,1}([0,1])$) and $n\in \mathcal{M}$. Using the fundamental solution of the second equation and the last two boundary conditions of (\ref{eq:original_ode}), we can now define the characteristic function $d(k)$, whose zeros correspond to all special transmission eigenvalues (see \cite{CCH16}):
\begin{equation}\label{eq:d(k)}
	d(k) = \det \left|
	\begin{matrix}
		y(1) & -\frac{\sin k}{k} \\
		y'(1) & -\cos k
	\end{matrix}
	\right|.
\end{equation}

In traditional inverse problems, the refractive index is typically assumed to satisfy $n \in C^2[0,1]$, where the Liouville transform serves as a powerful analytical tool to establish uniqueness result. However, in this work, we consider refractive indices possessing either $C_p^2[0,1]$ regularity or $C^{1,1}[0,1]$ smoothness. This broader regularity class necessitates more sophisticated technical tools to establish uniqueness results.

Furthermore, when $n\in C_p^2[0,1]$, we assume an $L$-segment refractive index defined by:
\begin{equation}\label{eq:n_piecewise_C2_def}
	n(r)=\left\{
	\begin{aligned}
		&n_1(r),\quad R_0 \leq r < R_1,\\
		&n_2(r),\quad R_1 \leq r < R_2,\\
		&\vdots\\
		&n_{L}(r),\quad R_{L-1} \leq r \leq R_{L},
	\end{aligned}
	\right.
\end{equation}
where $R_0=0$ and $R_L=1$. Denote by $\hat{\delta}_l=\int_{R_{l-1}}^{R_l}\sqrt{n_{l}(r)}\mathrm{d}r$ for $l=1,2,\cdots,L$. We denote
\begin{equation*}
	\tilde{\beta}_{\mathbf{i}}=1+\sum\limits_{l=1}^{L}(-1)^{i_l}\hat{\delta}_l,
\end{equation*}
where $\mathbf{i} = (i_1, \cdots, i_L)$ is a binary vector (i.e., $i_l \in \{0, 1\}$ for $l = 1, \cdots, L$). Each $\mathbf{i}$ corresponds to a decimal value $d(\mathbf{i}) = i_L \cdot 2^{L-1} + \cdots + i_1 \cdot 2^0 \in \{0, \cdots, 2^L - 1\}$.
To simplify the notation, we reindex the coefficients $\{\alpha_{\mathbf{i}}\}$ by a single subscript $j \in \{1, \cdots, 2^L\}$. The new index $j$ is defined by the transformation:
\begin{equation*}
j(\mathbf{i}) = d(\mathbf{i})+ 1.
\end{equation*}
The reordered sequence is denoted $\{\tilde{\beta}_j\}_{j=1}^{2^L}$.

Let $\varepsilon > 0$, define $\delta_L = \int_{0}^{1} \sqrt{n(r)}\, dr$, and denote
\begin{equation*}
	\varepsilon_0=\max_{1\leq j\leq 2^L}\left|n_{j}(R_{j-1})-n_{j-1}(R_{j-1})\right|.~
\end{equation*}
The main results of this work are as follows.
\begin{theorem}\label{thm:uniqueness}
	Assume that $n\in C_p^2[0,1]$ with $L$ intervals and $n\in \mathcal{M}$. Assume further that on each interval, it satisfies $0 < n_* < n(r) < n^*$,~$ |n'(r)|<\tilde{n}^{*}$,~$|n''(r)|<\tilde{n}^{**}$. Suppose the values $\{\hat{\delta}_l\}_{l=1}^{L}$ and the value of $n$ on $[\alpha, 1]$ are known, and $n(1)\neq 1$. Here, the point $\alpha$ is defined by the relation $\int_{\alpha}^{1} \sqrt{n(t)}  \mathrm{d}t = (\delta_L - 1 + \varepsilon)/2$ when $\delta \geq 1$, and $\alpha =1$ when $\delta_L <1$.
	
	If one of the following conditions holds:
	\begin{enumerate}
		\item[(1)] $\tilde{\beta}_{1}$ and $\tilde{\beta}_{2^L}$ are rationally linearly independent, and $\varepsilon_0\leq \frac{n_{*}M_1}{24M_2^{L-1}}$;
		\item[(2)] $\tilde{\beta}_{1}$ and $\tilde{\beta}_{2^L}$ are rationally linearly dependent, i.e., 
		$\tilde{\beta}_{2^{L}} = \dfrac{\hat{q}}{\hat{p}}\tilde{\beta}_1$ for some integers 
		$\hat{q}, \hat{p} \in \mathbb{Z}$ with $0 < |\hat{q}| < \hat{p}$, and one of the following holds:
         \begin{enumerate}
           \item $\delta_L>1$,~$n_L(1)>1$ (or $\delta_L<1$,~$n_L(1)<1$), and $\varepsilon_0\leq \frac{M_1\widetilde{M}_1n_{*}}{12M_2^{L-1}}$,
           \item $\delta_L<1$,~$n_L(1)>1$ (or $\delta_L>1$,~$n_L(1)<1$), and $\varepsilon_0\leq \frac{n_{*}M_1}{24M_2^{L-1}}$,
         \end{enumerate}
		where
		\begin{equation*}
        \begin{aligned}
        &M_1=\left|\left[n_L(1)\right]^{1/4}-\frac{1}{\left[n_L(1)\right]^{1/4}}\right|, \quad
		M_2=\max\left\{\left(\frac{n_*}{n^*}\right)^{\frac{1}{4}} +\left(\frac{n^*}{n_*}\right)^{\frac{1}{4}},\widetilde{M}_2\right\},\\
        &\widetilde{M}_1=\min\left\{\left|\sin\left(\frac{\hat{p}-|\hat{q}|}{\hat{p}+|\hat{q}|}\pi+2\hat{p}\varepsilon_1\right)\right|,~\left|\sin\left(\frac{\hat{p}-|\hat{q}|}{\hat{p}+|\hat{q}|}\pi-2\hat{p}\varepsilon_1\right)\right|\right\},~\widetilde{M}_2=\left[n_L(1)\right]^{\frac{1}{4}}+\frac{1}{\left[n_L(1)\right]^{\frac{1}{4}}};\\
        \end{aligned}
    	\end{equation*}
	\end{enumerate}
then the special transmission eigenvalues uniquely determine $n$.
\end{theorem}
The schematic diagram of the refractive index in Theorem \ref{thm:uniqueness} is shown in Figure 1.
\begin{figure}[htbp]\label{fig:n}
  \centering
  \includegraphics[width=0.8\textwidth]{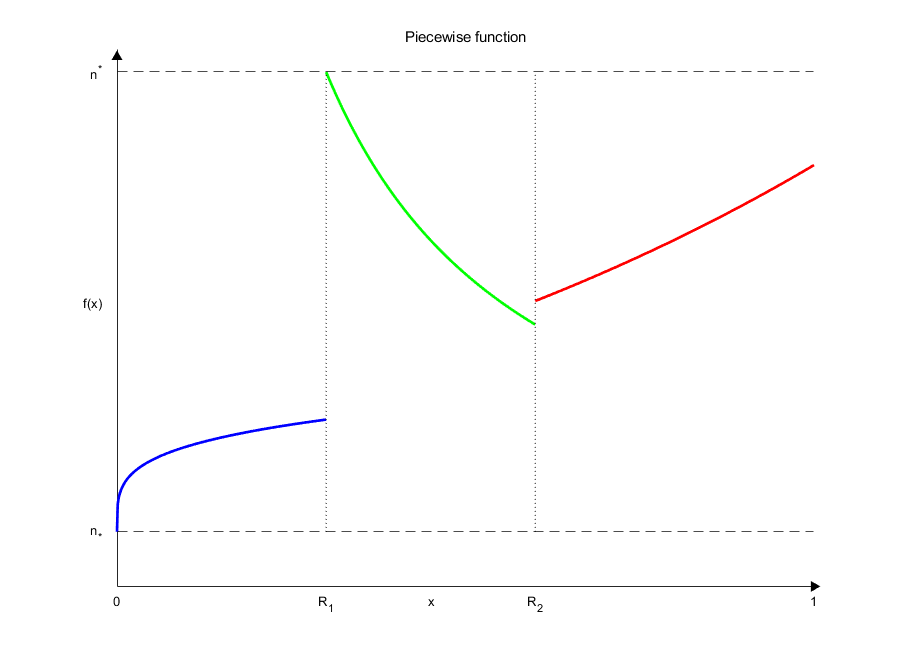}
  \caption{Diagram of the refractive index in Theorem 2.1.}
\end{figure}

\begin{theorem}\label{thm:uniqueness_C1}
	Assume that $n \in C^{1,1}[0,1] \cap \mathcal{M}$ and that it satisfies the bounds
	$0 < n_* < n(r), n'(r)< n^*$ for all $r \in [0,1]$. If the profile of $n$ on the interval $[\alpha, 1]$ are known, where $\alpha$ is as defined in Theorem~\ref{thm:uniqueness}, and if $n$ satisfies the boundary conditions $n(1) \neq 1$ and $n'(1) = 0$, then the special transmission eigenvalues uniquely determine $n$.
\end{theorem}

\begin{remark}\label{rem:low regularity counterexample}
If $n$ is a piecewise continuous function, then the special transmission eigenvalues cannot uniquely determine $n$. This demonstrates that reduced regularity of the refractive index increases the amount of data required for its unique determination in the inverse spectral problem. We will prove this claim in the next section by constructing an explicit counterexample.
\end{remark}

\begin{remark}\label{rem:theorem uniqueness_C1}
	In Theorem~\ref{thm:uniqueness}, we consider refractive indices that $n\in C_p^2[0,1]$. This mathematical formulation corresponds to multilayer structures commonly encountered in engineering applications, such as optical coatings or acoustic filters, thereby lending the result significant practical relevance. In practice, however, refractive indices may exhibit more complex variations. Notably, reduced smoothness requirements for the refractive index broaden the applicability of uniqueness results to more general scenarios. Theorem~\ref{thm:uniqueness_C1} extends this line of investigation by relaxing the smoothness requirement to $C^{1,1}[0,1]$ refractive indices.

The reduction in smoothness of the refractive index implies, mathematically speaking, that less spectral information is available for the inverse problem. Consequently, uniquely determining the refractive index necessarily requires additional data. In our theorem, the requirement that the refractive index exhibits only small jumps across layer interfaces corresponds physically to the condition that adjacent layers have relatively close material properties---an assumption naturally satisfied in many practical engineering applications.
\end{remark}

In the remainder of this section, we recall a classical result from Diophantine approximation that will be essential for our subsequent arguments.
\begin{lemma}[{\cite{HW80}}]\label{lem:Kronecker_approximation}
	If $\{\mathfrak{v}_j\}_{j=1}^{n}$ are linearly independent over $\mathbb{Q}$, and $\{\alpha_j\}_{j=1}^{n}$, $T$, and $\varepsilon_1$ are positive real numbers, then there exists a real number $t > T$ and integers $\{p_j\}_{j=1}^{n}$ such that
	\begin{equation*}
		\left| t \mathfrak{v}_j - p_j - \alpha_j \right| < \varepsilon_1, \quad j = 1, 2, \cdots, n.
	\end{equation*}
\end{lemma}
By virtue of this lemma, we need only impose the condition $0 < n_* < n(r) < n^{*}$ for $r \in [0,1]$, in contrast to previous approaches that assume either $0 < n < 1$ or $n > 1$ (see \cite{CCH16,CL13}).


\section{A counterexample of uniqueness and preliminaries}\label{sec:A counterexample and preliminaries}


\subsection{A counterexample of uniqueness}\label{sec:counterexample}

First, we illustrate Remark\,\ref{rem:low regularity counterexample} with an explicit counterexample.

Consider the case where the function $n$ is piecewise constant with two segments and a discontinuity at $R_1$; that is,	
\begin{equation*}
	n(r) = \begin{cases}
		n_1, & 0 \leq r < R_1, \\
		n_2, & R_1 \leq r < 1,
	\end{cases}
	\nonumber
\end{equation*}
where $n_1$ and $n_2$ are two given constants. In this case, the solution to the governing equation
\begin{equation*}
	\left\{
	\begin{aligned}
		y'' + k^2 n y = 0, & \quad r \in (0, 1), \\
		y(0) = 0, & \quad y'(0) = 1,
	\end{aligned}
	\nonumber
	\right.
\end{equation*}
can be expressed as
\begin{equation}\label{eq:piecewise_constant_ode_solution}
	y(r) = \begin{cases}
	\displaystyle	\frac{\sin(k\sqrt{n_1}r)}{k\sqrt{n_1}}, & 0 \leq r < R_1, \\
	\displaystyle	b_{21} \frac{\sin(k\sqrt{n_2}(r - R_1))}{k\sqrt{n_2}}
            + b_{22} \cos(k\sqrt{n_2}(r - R_1)), & R_1 \leq r < 1.
	\end{cases}
\end{equation}	
The continuity of $y$ at $R_1$ imposes the condition:
\begin{equation}\label{eq:linear_combination_coefficients}
		b_{21} = \cos(k\sqrt{n_1} R_1), \quad
		b_{22} = \frac{\sin(k\sqrt{n_1} R_1)}{k\sqrt{n_1}}.
\end{equation}
Substitution of \eqref{eq:piecewise_constant_ode_solution} and \eqref{eq:linear_combination_coefficients} into \eqref{eq:d(k)} yields	
\begin{eqnarray}\label{eq:piecewise_constant_d(k)}
		d(k) &=& \left\{ \cos(k\sqrt{n_1} R_1) \cos(k\sqrt{n_2}(1 - R_1)) + \frac{\sqrt{n_2} \sin(k\sqrt{n_1} R_1)}{\sqrt{n_1}} \sin(k\sqrt{n_2}(1 - R_1)) \right\} \frac{\sin k}{k} \nonumber\\
		&- & \left\{ \cos(k\sqrt{n_1} R_1) \frac{\sin(k\sqrt{n_2}(1 - R_1))}{k\sqrt{n_2}} + \frac{\sin(k\sqrt{n_1} R_1)}{k\sqrt{n_1}} \cos(k\sqrt{n_2}(1 - R_1)) \right\} \cos k.
\end{eqnarray}

We now define two specific piecewise constant functions:	
\begin{equation}\label{eq:n^{(1)},n^{(2)}}
	n^{(1)}(r) = \begin{cases}
		4, & 0 \leq r < \frac{1}{2}, \\
		16, & \frac{1}{2} \leq r \leq 1,
	\end{cases}
	\quad
	n^{(2)}(r) = \begin{cases}
		16, & 0 \leq r < \frac{1}{2}, \\
		4, & \frac{1}{2} \leq r \leq 1.
	\end{cases}
\end{equation}
Let $d^{(1)}(k)$ and $d^{(2)}(k)$ denote the characteristic functions corresponding to $n^{(1)}$ and $n^{(2)}$, respectively. Substituting the expressions from \eqref{eq:n^{(1)},n^{(2)}} into \eqref{eq:piecewise_constant_d(k)} yields
\begin{equation*}
	d^{(1)}(k) = 3 d^{(2)}(k)  = \frac{9}{16k} (\sin(2k) + \sin(4k)).
\end{equation*}
This relation implies that the zeros of $d^{(1)}(k)$ and $d^{(2)}(k)$ are identical. By the definition in \eqref{eq:d(k)}, this means the associated transmission eigenvalues for $n^{(1)}$ and $n^{(2)}$ also coincide. This provides an explicit counterexample demonstrating the non-uniqueness of the solution of inverse problem.

This non-uniqueness shows that relaxing the regularity of $n$ to piecewise constant prevents its unique determination from special transmission eigenvalues. Therefore, uniquely reconstructing $n$ necessitates additional information.


\subsection{Liouville transform}\label{sec:Liouville transform}

We now introduce the Liouville transform on each interval. Given that $n$ is defined by \eqref{eq:n_piecewise_C2_def}, the solution to \eqref{eq:original_ode} admits the representation
\begin{equation}\label{eq:general_form_ode_solution_y}
	y(r)=\left\{
	\begin{aligned}
		&b_1 g_1(r), \quad R_0 \leq r < R_1, \\
		&b_{l1} g_{l1}(r) + b_{l2} g_{l2}(r), \quad  R_{l-1} \leq r< R_l, \quad 2 \leq l \leq L-1,\\
		&b_{L1} g_{L1}(r) + b_{L2} g_{L2}(r), \quad  R_{L-1} \leq r \leq R_L,
	\end{aligned}
	\right.
\end{equation}
where the functions $g_{l1}$ and $g_{l2}$ are defined as the solutions on the interval $(R_{l-1}, R_l)$ of the initial value problems
\begin{equation}\label{eq:ode_y_n_j}
	\left\{
	\begin{aligned}
		&y'' + k^2 n_l(r) y = 0, \quad  R_{l-1} <r< R_l, \\
		&y(R_{l-1}) = 0, \quad y'(R_{l-1}) = 1,
	\end{aligned}
	\right.
	\quad \mbox{and} \quad
	\left\{
	\begin{aligned}
		&y'' + k^2 n_l(r) y = 0, \quad  R_{l-1} < r< R_l, \\
		&y(R_{l-1}) = 1, \quad y'(R_{l-1}) = 0,
	\end{aligned}
	\right.
\end{equation}
respectively.

To simplify these expressions, we apply the Liouville transform. Let $\delta_0 = 0$ and define Liouville transform for each layer $l=1,2,\cdots,L$:
\begin{equation}\label{eq:Liouville transform}
	\left\{
	\begin{aligned}
		&\xi_l(r) = \int_{R_{l-1}}^r \sqrt{n_l(r)}\,dr+\delta_{l-1} , \quad R_{l-1} \leq r < R_l,\\
		&z(\xi_l(r)) = [n_l(r)]^{\frac{1}{4}} y(r), \quad \delta_{l-1} \leq \xi_l < \delta_l,\\
		&\delta_l = \delta_{l-1} + \int_{R_{l-1}}^{R_l} \sqrt{n_l(r)} \, dr.
	\end{aligned}
	\right.
\end{equation}
Applying these transformations to the equations in \eqref{eq:ode_y_n_j}, the solution \eqref{eq:general_form_ode_solution_y} becomes
\begin{equation}\label{eq:original_ode_solution_y}
	y(r) =
	\left\{
	\begin{aligned}
		&b_1 \frac{z_1(\xi_1(r))}{[n_{1}(r)]^{\frac{1}{4}}}, \quad R_0 \leq r< R_1, \\
		&b_{l1} \frac{z_{l1}(\xi_l(r))}{[n_l(r)]^{\frac{1}{4}}} +b_{l2}\frac{z_{l2}(\xi_l(r))}{[n_l(r)]^{\frac{1}{4}}},  \quad R_{l-1} \leq r<R_l, \quad 2 \leq l \leq L-1, \\
		&b_{L1} \frac{z_{L1}(\xi_L(r))}{[n_L(r)]^{\frac{1}{4}}} +b_{L2}\frac{z_{L2}(\xi_L(r))}{[n_L(r)]^{\frac{1}{4}}},  \quad R_{L-1} \leq r \leq R_L.
	\end{aligned}
	\right.
\end{equation}
Here, $z_{l1}$ and $z_{l2}$ are solutions to the transformed differential equation on $(\delta_{l-1}, \delta_l)$:
\begin{equation*}\label{eq:basis_for_z_j}
	z'' + \left(k^2 - p(\xi_l)\right) z = 0,
\end{equation*}
with the potential $p(\xi_l)$ given by
\begin{equation*}
p(\xi_l) = \frac{n_l''(r)}{4n_l^2(r)} - \frac{5[n_l'(r)]^2}{16n_l^3(r)}, \quad R_{l-1} \leq r<R_l,
\end{equation*}
and satisfy the initial conditions:
\begin{align*}
	&\text{For $z_{l1}$:} & z(\delta_{l-1}) &= 0, & z'(\delta_{l-1}) &= n_l^{-1/4}(R_{l-1}); \\
	&\text{For $z_{l2}$:} & z(\delta_{l-1}) &= n_l^{1/4}(R_{l-1}), & z'(\delta_{l-1}) &= \dfrac{n_l'(R_{l-1})}{4n_l^{3/2}(R_{l-1})}.
\end{align*}
We note that the length of the transformed interval is $\hat{\delta}_l =(\delta_l - \delta_{l-1})$ for $l = 1, \cdots, L$. Furthermore, it follows from \cite{PT87} that
\begin{equation}\label{eq:z1z2 dom}	
	\left\{
	\begin{aligned}
		&z_1(\xi_1)= \frac{1}{[n_{1}(R_{0})]^{\frac{1}{4}}k}\sin{(\xi_1)}+O\left(\frac{1}{|k|^2}\right),\\
		&z_{l1}(\xi_l)= \frac{1}{[n_{l}(R_{l-1})]^{\frac{1}{4}}k}\sin{(\xi_{l}-\delta_{l-1})}+O\left(\frac{1}{|k|^2}\right), \quad 2 \leq l \leq L, \\
		&z_{l2}(\xi_l)= [n_{l}(R_{l-1})]^{\frac{1}{4}}\cos{(\xi_{l}-\delta_{l-1})}+O\left(\frac{1}{|k|}\right), \quad 2 \leq l \leq L.
	\end{aligned}
	\right.
\end{equation}


\subsection{Key functions}\label{sec:key functions}

We introduce a key relation that will be used in the subsequent proofs of Theorems \ref{thm:uniqueness} and \ref{thm:uniqueness_C1}.

Suppose two distinct refractive indices, $n^{(1)}$ and $n^{(2)}$, yield the same set of special transmission eigenvalues. Then, for a given eigenvalue $k$, the corresponding functions $y_1$ and $y_2$ (for $n^{(1)}$ and $n^{(2)}$, respectively) satisfy the following system for $j = 1, 2$:
\begin{equation}\label{eq:y_1,y_2a}
	\left\{
	\begin{aligned}
		y_j^{\prime\prime} + k^{2} n^{(j)}(r) y_j &= 0, \quad &0 < r < 1,\\
		y_0^{\prime\prime} + k^{2} y_0 &= 0, \quad &0 < r < 1,\\
		y_j(1) &= y_0(1), \\
		y_j'(1) &= y_0'(1).
	\end{aligned}
	\right.
\end{equation}
Multiplying the first equation for $j=1$ by $y_2$ and the first equation for $j=2$ by $y_1$, then subtracting the results yields:
\begin{equation*}
	\begin{aligned}
		y_1''(r) y_2(r) - y_1(r) y_2''(r) &= k^2 \left[n^{(2)}(r) - n^{(1)}(r)\right] y_1(r) y_2(r).
	\end{aligned}
\end{equation*}
Integrating both sides from $0$ to $1$ gives:
\begin{equation*}
	\begin{aligned}
		k^2 \int_{0}^{1} \left[n^{(2)}(r) - n^{(1)}(r)\right] y_1(r) y_2(r) \, \mathrm{d} r &= \left\{y_1'(1)y_2(1) - y_1(1)y_2'(1)\right\} - \left\{y_1'(0)y_2(0) - y_1(0)y_2'(0)\right\} \text{.}
	\end{aligned}
\end{equation*}
From the priori condition that the values of $n(r)$ are given for $r \in (\alpha,1)$, it follows that $n_1(r)=n_2(r)$ for $r \in (\alpha,1)$. Therefore, the upper limit of the integral on the left-hand side should be $\alpha$.
The right-hand side is the Wronskian evaluated at the endpoints. Since the solution $y$ of the transmission eigenvalue problem \eqref{eq:transmission eigenvalue problem} belongs to $H^1(D)$, its radial representation \eqref{eq:radial solutions} implies $y_1(0) = y_2(0)$.~Therefore, the Wronskian vanishes at $r=0$, and we obtain the fundamental relation:
\begin{equation*}
	\begin{aligned}
		k^2 \int_{0}^{\alpha} \left[n^{(2)}(r) - n^{(1)}(r)\right] y_1(r) y_2(r) \, \mathrm{d} r &= \left\{y_1'(1)y_2(1) - y_1(1)y_2'(1)\right\}.
	\end{aligned}
\end{equation*}

Now, we define two key functions:
\begin{equation}\label{eq:def f g}
f(k):=k^2 \int_{0}^{\alpha} \left[n^{(2)}(r) - n^{(1)}(r)\right] y_1(r) y_2(r) \, \mathrm{d} r,
\quad		g(k):=\frac{f(k)}{d(k)},	
\end{equation}
where $d$ is defined as in \eqref{eq:d(k)}.


\section{Proof of Theorem \ref{thm:uniqueness}}\label{sec:thm2.1}
In this section, we present the proof of Theorem~\ref{thm:uniqueness}. The proof strategy involves establishing several lemmas that provide estimates for the solution $y$ of~\eqref{eq:original_ode} and for the characteristic function $d$ defined in~\eqref{eq:d(k)}. These estimates hold under the assumption that $n\in C^2_p[0,1]$. We also require some fundamental integral identities.

We begin by analyzing the solution $y$ of~\eqref{eq:original_ode}. In the case where  $n\in C^2_p[0,1]$, the Liouville transform can be applied to~\eqref{eq:original_ode}, reducing it to a standard Sturm-Liouville equation (\cite{CCH16}). For such equations, the estimation of solutions is well-established in \cite{PT87}, with further details available in \cite{AGP11}.


The following lemma extends this result to $n\in C^2_p[0,1]$ with $L$ intervals, which can be proven by induction on $L$.
\begin{lemma}\label{lem:y,y'}
Assume $y$ is a solution of \eqref{eq:original_ode} $n\in C^2_p[0,1] \cap \mathcal{M}$.
Assume further that on each interval, it satisfies $0 < n_* < n(r), n'(r), n''(r) < n^*$.
Then for $|k| \geq 1 $ and $0 \leq r \leq 1$, the following estimates hold:
	\begin{equation}\label{eq:y}
		\left| y(r) \right| \leq C \left( \frac{1}{|k|^2} + \frac{1}{|k|} \right)\exp\left( \left| \Im k \right| \int_{0}^{r} \sqrt{n(t)} \, dt \right),
	\end{equation}
	\begin{equation}\label{eq:y'}
		\left| y'(r) \right| \leq C \left( \frac{1}{|k|} + 1 \right)\exp\left( \left| \Im k \right| \int_{0}^{r} \sqrt{n(t)} \, dt \right),
	\end{equation}	
where $C$ depends only on $L$ and $n$, and $\Im k$ denotes the imaginary part of $k$.
\end{lemma}

\begin{proof}
We proceed by induction on $L$. The base case $L=1$ is established in \cite{AGP11}.
	
Assume that the estimates hold for $L = m$. We now prove them for $L = m+1$.
First, the refractive index $n$ is defined as in \eqref{eq:n_piecewise_C2_def},
with the parameter $L$ in this definition set to $m+1$. Next, we perform the Liouville transformations in Section\,\ref{sec:Liouville transform}. By the induction hypothesis, the estimates \eqref{eq:y} and \eqref{eq:y'} hold for $y$ on the interval $0 \leq r < R_m$. On the interval $R_m \leq r \leq R_{m+1}$, we have
\begin{equation}\label{eq:yr eqn}
	y(r) = b_{(m+1)1} \frac{z_{(m+1)1}(\xi_{m+1}(r))}{[n_{m+1}(r)]^{\frac{1}{4}}}
   + b_{(m+1)2} \frac{z_{(m+1)2}(\xi_{m+1}(r))}{[n_{m+1}(r)]^{\frac{1}{4}}}.
\end{equation}

On one hand, the regularity requirements of the interior transmission problem \cite{CCH16} dictate that solutions of \eqref{eq:original_ode} must be at least continuously differentiable at each interface $R_l$ for $l=1,\cdots,m$. This yields the matching conditions:
\begin{equation*}
	\left\{
	\begin{aligned}
		&b_{l1} g_{l1}(R_l) + b_{l2} g_{l2}(R_l) =b_{(l+1)2} g_{(l+1)2}(R_{l}), \\
		&b_{l1} g'_{l1}(R_l) + b_{l2} g'_{l2}(R_l) =b_{(l+1)1} g'_{(l+1)1}(R_{l}).
	\end{aligned}
	\right.
\end{equation*}
Combined with the initial conditions from \eqref{eq:ode_y_n_j}, we obtain the recurrence relations:
\begin{equation}\label{eq:equation_b_j_c_j}
	\left\{
	\begin{aligned}
		&b_{(l+1)1} = b_{l1} g'_{l1}(R_l) + b_{l2} g'_{l2}(R_l), \\
		&b_{(l+1)2} = b_{l1} g_{l1}(R_l) + b_{l2} g_{l2}(R_l),
	\end{aligned}
	\right.
\end{equation}
with initial values $b_{11}=b_{1}=1$ and $b_{12}=0$. From \eqref{eq:equation_b_j_c_j} and \eqref{eq:general_form_ode_solution_y}, it follows that:
\begin{equation}\label{eq:b_l+11_b_l+12}
	b_{(l+1)1}=y'(R_l), \quad b_{(l+1)2}=y(R_l), \quad l=1,\cdots,m.
\end{equation}
Applying the inductive hypothesis, we conclude:
\begin{equation}\label{eq:bestimation}
	\left\{	
	\begin{aligned}
		&\left|b_{(m+1)1}\right| <C \left( \frac{1}{|k|} + 1 \right)\exp\left( \left| \Im k \right| \int_{0}^{R_m} \sqrt{n(t)} \, dt \right), \\
		&\left|b_{(m+1)2}\right| < C \left( \frac{1}{|k|^2} + \frac{1}{|k|} \right)\exp\left( \left| \Im k \right| \int_{0}^{R_m} \sqrt{n(t)} \, dt \right),
	\end{aligned}
	\right.
\end{equation}
where $C$ is a constant independent of $k$.~
	
On the other hand, Theorem 3 in Chapter 1 of the monograph \cite{PT87} yields
\begin{equation}\label{eq:zestimation}
	\left\{
	\begin{aligned}
		& \left|z_{(m+1)1}\right| \leq C \left\{ \frac{1}{|k|^2} + \frac{1}{|k|} \right\}\exp\left(|\Im k| (\xi_{m+1} - \delta_{m})\right), \\
		& \left|z_{(m+1)2}\right| \leq C \left\{ \frac{1}{|k|}+1 \right\}\exp\left(|\Im k| (\xi_{m+1} - \delta_{m})\right),
	\end{aligned}
	\right.
\end{equation}
where $C$ is a constant independent of $k$.
	
Combining the estimates \eqref{eq:bestimation} and \eqref{eq:zestimation} with equation \eqref{eq:yr eqn} and the condition $0 < n* < n < n^*$ yields the result in \eqref{eq:y}. The estimate \eqref{eq:y'} can be established similarly.
\end{proof}


To evaluate $d(k)$ using \eqref{eq:d(k)}, we must first identify the dominant terms in the asymptotic expansions of $y(1)$ and $y'(1)$. For sufficiently large $k$, we denote these leading terms by $D(y(1))$ and $D(y'(1))$, respectively.

\begin{lemma}\label{lem:y(1)_adequate_expression}
The dominant terms of $y(1)$ and $y'(1)$ are given by
\begin{equation}
D(y(1))=\frac{1}{\left[n_L(1)\right]^{\frac{1}{4}}k}\sum_{j=1}^{2^{L-1}}A_j\sin{\beta_j k},  \quad \quad D(y'(1))=\left[n_L(1)\right]^{\frac{1}{4}}\sum_{j=1}^{2^{L-1}}A_j\cos{\beta_j k},
\end{equation}
where
\begin{eqnarray*}
A_{j}&=&\frac{1}{2^{L-1}}\frac{1}{\left[n_{1}(0)\right]^{\frac{1}{4}}}
\prod\limits_{l=2}^{L}\left(\frac{\left[n_{l-1}(R_{l-1})\right]^{\frac{1}{4}}}{\left[n_l(R_{l-1})\right]^{\frac{1}{4}}}+(-1)^{i'_l} \frac{\left[n_l(R_{l-1})\right]^{\frac{1}{4}}}{\left[n_{l-1}(R_{l-1})\right]^{\frac{1}{4}}} \right), \\
\beta_j&=&\sum\limits_{l=1}^{L-1}(-1)^{i_l}\hat{\delta}_l+\delta_{L},\quad j=1,\cdots,2^{L-1}.
\end{eqnarray*}
Here, the relation between $i'_l$ and $i_l$ is that ${i'_l}={i_l + i_{l-1}}(\hspace{-2mm}\mod 2)$ for $2 \le l \le L$, where $i_L=0$.
\end{lemma}

\begin{proof}
We prove this lemma by induction.

For the base case $L=2$, \eqref{eq:original_ode_solution_y} implies
\begin{equation}\label{eq:y(1)}
y(1)=b_{21}\frac{z_{21}\left(\xi_2(1)\right)}{\left[n_{2}(1)\right]^{\frac{1}{4}}}+b_{22}\frac{z_{22}(\xi_2(1))}{\left[n_{2}(1)\right]^{\frac{1}{4}}}.
\end{equation}

From \eqref{eq:b_l+11_b_l+12}, it is clear that $b_{21}=y'(R_1)$ and $b_{22}=y(R_1)$. Together with \eqref{eq:Liouville transform}, this yields
\begin{equation}\label{eq:Db21}
	D(b_{21})=\frac{\left[n_1(R_1)\right]^{\frac{1}{4}}}{\left[n_1(0)\right]^{\frac{1}{4}}}\cos\left(k\hat{\delta}_1\right),\quad
	D(b_{22})=\frac{1}{\left[n_1(0)\right]^{\frac{1}{4}}\left[n_1(R_1)\right]^{\frac{1}{4}}}\frac{\sin\left(k\hat{\delta}_1\right)}{k},
\end{equation}
respectively. It follows from (\ref{eq:z1z2 dom}) that
\begin{equation}\label{eq:Dz21}
	D(z_{21}\left(\xi_2(1)\right))=\frac{1}{\left[n_{2}(R_1)\right]^{\frac{1}{4}}}\frac{\sin\left(k\hat{\delta}_2\right)}{k},\quad
	D(z_{22}\left(\xi_2(1)\right))=\left[n_{2}(R_1)\right]^{\frac{1}{4}}\cos\left(k\hat{\delta}_2\right).
\end{equation}
Combining the estimates \eqref{eq:Db21} and \eqref{eq:Dz21} with equation \eqref{eq:y(1)} yields
\begin{eqnarray*}
D(y(1))&=&\frac{1}{2k}\frac{1}{\left[n_{1}(0)\right]^{\frac{1}{4}}\left[n_{2}(1)\right]^{\frac{1}{4}}}\left(\frac{\left[n_{1}(R_1)\right]^{\frac{1}{4}}}{\left[n_{2}(R_1)\right]^{\frac{1}{4}}}+\frac{\left[n_{2}(R_1)\right]^{\frac{1}{4}}}{\left[n_{1}(R_1)\right]^{\frac{1}{4}}}\right)\sin\left\{k\left(\hat{\delta}_2+\hat{\delta}_1\right)\right\}\\
		&+&\frac{1}{2k}\frac{1}{\left[n_{1}(0)\right]^{\frac{1}{4}}\left[n_{2}(1)\right]^{\frac{1}{4}}}\left(\frac{\left[n_{1}(R_1)\right]^{\frac{1}{4}}}{\left[n_{2}(R_1)\right]^{\frac{1}{4}}}-\frac{\left[n_{2}(R_1)\right]^{\frac{1}{4}}}{\left[n_{1}(R_1)\right]^{\frac{1}{4}}}\right)\sin\left\{k\left(\hat{\delta}_2-\hat{\delta}_1\right)
		\right\}.
\end{eqnarray*}
Similarly, we have
\begin{eqnarray*}
D(y'(1))&=& \frac{1}{2}\frac{\left[n_{2}(1)\right]^{\frac{1}{4}}}{\left[n_{1}(0)\right]^{\frac{1}{4}}}\left(\frac{\left[n_{1}(R_1)\right]^{\frac{1}{4}}}{\left[n_{2}(R_1)\right]^{\frac{1}{4}}}+\frac{\left[n_{2}(R_1)\right]^{\frac{1}{4}}}{\left[n_{1}(R_1)\right]^{\frac{1}{4}}}\right)\cos\left\{k\left(\hat{\delta}_2+\hat{\delta}_1\right)\right\}\\
		&+&\frac{1}{2}\frac{\left[n_{2}(1)\right]^{\frac{1}{4}}}{\left[n_{1}(0)\right]^{\frac{1}{4}}}\left(\frac{\left[n_{1}(R_1)\right]^{\frac{1}{4}}}{\left[n_{2}(R_1)\right]^{\frac{1}{4}}}-\frac{\left[n_{2}(R_1)\right]^{\frac{1}{4}}}{\left[n_{1}(R_1)\right]^{\frac{1}{4}}}\right)\cos\left\{k\left(\hat{\delta}_2-\hat{\delta}_1\right)\right\}.
\end{eqnarray*}
It is immediately that the conclusion of this lemma holds for $L=2$.

Assume that the conclusions hold for $L=m$. We now prove them for $L=m+1$.
It follows from \eqref{eq:original_ode_solution_y} that
\begin{eqnarray}\label{eq:y1m+1}
		y(1)&=&b_{(m+1)1}\frac{z_{(m+1)1}\left(\delta_{m+1}(r)\right)}{\left[n_{m+1}(1)\right]^{\frac{1}{4}}}+b_{(m+1)2}\frac{z_{(m+1)2}(\delta_{m+1}(r))}{\left[n_{m+1}(1)\right]^{\frac{1}{4}}}.
\end{eqnarray}
From \eqref{eq:b_l+11_b_l+12}, there is:
\begin{equation*}
	b_{(m+1)1}=y'(R_m), \quad b_{(m+1)2}=y(R_m), 
\end{equation*}
The induction hypothesis for $L=m$ implies
\begin{equation}\label{eq:Dbm1}
	D(b_{(m+1)1})=\sum_{j=1}^{2^{m-1}}A_j\left[n_m(R_m)\right]^{\frac{1}{4}}\cos{\beta_j k},\quad
	D(b_{(m+1)2})=\sum_{j=1}^{2^{m-1}}\frac{A_j}{\left[n_m(R_m)\right]^{\frac{1}{4}}}\frac{\sin{\beta_j k} }{k},
\end{equation}
respectively. It follows from (\ref{eq:z1z2 dom}) that
\begin{equation}\label{eq:Dzm1}
\hspace{-5mm}	D(z_{(m+1)1}\left(\xi_{(m+1)}(1)\right))=\frac{1}{\left[n_{m+1}(R_{m})\right]^{\frac{1}{4}}}\frac{\sin\left(k\hat{\delta}_{m+1}\right)}{k},~
	D(z_{(m+1)2}\left(\xi_{(m+1}(1)\right))=\left[n_{m+1}(R_{m})\right]^{\frac{1}{4}}\cos\left(k\hat{\delta}_{m+1}\right).
\end{equation}
Combining the estimates \eqref{eq:Dbm1} and \eqref{eq:Dzm1} with equation \eqref{eq:y1m+1} yields
\begin{eqnarray*}
D(y(1))&=&\left( \sum_{j=1}^{2^{m-1}}A_j\left[n_m(R_m)\right]^{\frac{1}{4}}\cos{\beta_j k} \right)\frac{1}{\left[n_{m+1}\left(R_m\right)\right]^{\frac{1}{4}}\left[n_{m+1}(1)\right]^{\frac{1}{4}}}\frac{\sin{k\hat{\delta}}_{m+1}}{k}\\
		&+&\left( \sum_{j=1}^{2^{m-1}}\frac{A_j}{\left[n_m(R_m)\right]^{\frac{1}{4}}}\frac{\sin{\beta_j k} }{k} \right)\frac{\left[n_{m+1}\left(R_m\right)\right]^{\frac{1}{4}}}{\left[n_{m+1}(1)\right]^{\frac{1}{4}}}\cos{k\hat{\delta}_{m+1}}\\
		&=&\left\{\sum_{j=1}^{2^{m-1}}\frac{A_j}{2k\left[n_{m+1}(1)\right]^{\frac{1}{4}}}\frac{\left[n_m\left(R_m\right)\right]^{\frac{1}{4}}}{\left[n_{m+1}\left(R_m\right)\right]^\frac{1}{4}}\left[\sin\left\{k\left(\hat{\delta}_{m+1}+\beta_j\right)\right\}+\sin\left\{k\left(\hat{\delta}_{m+1}-\beta_j\right)\right\}\right]\right\}\\
		&+&\left\{\sum_{j=1}^{2^{m-1}}\frac{A_j}{2k\left[n_{m+1}(1)\right]^{\frac{1}{4}}}\frac{\left[n_{m+1}\left(R_m\right)\right]^{\frac{1}{4}}}{\left[n_{m}\left(R_m\right)\right]^\frac{1}{4}}\left[\sin\left\{k\left(\hat{\delta}_{m+1}+\beta_j\right)\right\}-\sin\left\{k\left(\hat{\delta}_{m+1}-\beta_j\right)\right\}\right]\right\}\\
		&=&\frac{1}{2k\left[n_{m+1}\left(1\right)\right]^{\frac{1}{4}}}\sum_{j=1}^{2^{m-1}}\left\{A_j\left(\frac{\left[n_{m}\left(R_m\right)\right]^{\frac{1}{4}}}{\left[n_{m+1}\left(R_m\right)\right]^\frac{1}{4}}+\frac{\left[n_{m+1}\left(R_m\right)\right]^{\frac{1}{4}}}{\left[n_{m}\left(R_m\right)\right]^\frac{1}{4}}\right)\sin\left\{ k\left(\hat{\delta}_{m+1}+\beta_j\right) \right\}\right.\\
		&+&\left.A_j\left(\frac{\left[n_{m}\left(R_m\right)\right]^{\frac{1}{4}}}{\left[n_{m+1}\left(R_m\right)\right]^\frac{1}{4}}-\frac{\left[n_{m+1}\left(R_m\right)\right]^{\frac{1}{4}}}{\left[n_{m}\left(R_m\right)\right]^\frac{1}{4}}\right)\sin\left\{k\left(\hat{\delta}_{m+1}-\beta_j\right)\right\}\right\}.
\end{eqnarray*}
It is immediate that the conclusion of this lemma for $y(1)$ holds for $L=m+1$.
The reasoning for $y'(1)$ for $L=m+1$ is the same. Therefore, we have completed the induction.
\end{proof}


Next, we estimate the characteristic function $d$. To do so, we first need to calculate the dominant term of $d(k)$.
\begin{lemma}\label{lem:main_term_d(k)}
The dominant term of $d(k)$ is given by
\begin{equation}\label{eq:full_main_term_dk}
	D(d(k)) = \frac{1}{2^L k} \sum_{j=1}^{2^L} \tilde{A}_j \sin(\tilde{\beta}_j k),
\end{equation}
where
\begin{eqnarray*}
	\tilde{A}_j &=& \left( \left[n_L(1)\right]^{1/4}
	+ (-1)^{i'_{L+1}} \frac{1}{\left[n_L(1)\right]^{1/4}} \right)
	\prod_{l=2}^{L} \left( \frac{\left[n_{l-1}(R_{l-1})\right]^{1/4}}{\left[n_l(R_{l-1})\right]^{1/4}}
	+ (-1)^{i'_l} \frac{\left[n_l(R_{l-1})\right]^{1/4}}{\left[n_{l-1}(R_{l-1})\right]^{1/4}} \right),\\
\tilde{\beta}_j &=& 1 + \sum\limits_{l=1}^{L} (-1)^{i_l} \hat{\delta}_l,\quad j=1,\cdots,2^{L}.
\end{eqnarray*}
The relation between $i'_l$ and $i_l$ is given by
\begin{equation*}
i'_{L+1} = \min \{ 1-i_{L+1}, 1-i_{L} \}
\quad \text{and} \quad
{i'_l} = {i_l + i_{l-1}}(\hspace{-3.5mm}\mod 2) \quad \text{for} \quad 2 \leq l \leq L,
\end{equation*}
where $i_{L+1}=0$.
\end{lemma}
\begin{proof}
Substituting the dominant terms of $y(1)$ and $y'(1)$ into the expression for $d(k)$ yields
\begin{equation*}
	\begin{aligned}
		D(d(k))&= \left\{\left[n_L(1)\right]^{\frac{1}{4}}\sum_{j=1}^{2^{L-1}}A_j\cos{\beta_j k} \right\}\frac{\sin{k}}{k}-\left\{\frac{1}{\left[n_L(1)\right]^{\frac{1}{4}}}\sum_{j=1}^{2^{L-1}}\frac{A_j\sin{\beta_j k} }{k} \right\}\cos{k}\\
&=\frac{1}{2k}\sum_{j=1}^{2^{L-1}}A_j\left\{\left(\left[n_L(1)\right]^{\frac{1}{4}}-\frac{1}{\left[n_L(1)\right]^{\frac{1}{4}}}\right)\sin{\left(1+\beta_j\right)k}+\left( \left[n_L(1)\right]^{\frac{1}{4}}+\frac{1}{\left[n_L(1)\right]^{\frac{1}{4}}}\right)\sin{\left(1-\beta_j\right)k} \right\},
	\end{aligned}
\end{equation*}
together with the expressions for $\beta_j$ and $A_j$ in Lemma\,\ref{lem:y(1)_adequate_expression} implies the conclusion.
\end{proof}


Consider two distinct refractive indices, $n^{(1)}(r)$ and $n^{(2)}(r)$. The corresponding initial value problems for $j = 1, 2$ yield:
\begin{equation}\label{eq:y_1,y_2b}
	\left\{
	\begin{aligned}
		&y_j^{\prime\prime} - k^{2} n^{(j)}(r) y_j = 0, \quad 0 < r < 1,\\
		&y_j(0) = 0, \quad y_j'(0) = 1.
	\end{aligned}
	\right.
\end{equation}
\begin{lemma}\label{lem:completeness}(\cite{Ram09})
Assume that for $j=1,2$, the coefficients satisfy $n^{(j)} \in \mathcal{M}$ with uniform bounds $0 < n_* \leq n^{(j)}(r) \leq n^*$, and that $y_j$ are solutions of (\ref{eq:y_1,y_2b}). If
\begin{equation*}
	\int_{0}^{1} q(r) y_1(r) y_2(r) \,\mathrm{d}r = 0, \quad \text{for all } k \in \mathbb{R},
\end{equation*}
and $q(r)\in \mathcal{M}$,~then $q(r) \equiv 0$.
\end{lemma}


Now, we are at the stage to prove Theorem \ref{thm:uniqueness}.
\begin{proof}
We first observe that for all special transmission eigenvalues, the value of $\tilde{\beta}_1 = 1 + \delta_L$ can be obtained from \eqref{eq:full_main_term_dk} and reference \cite{CCH16}. This immediately yields $\tilde{\beta}_{2^L} = 1 - \delta_L$, where $\delta_l$ (with $1 \leq l \leq L$) is defined in \eqref{eq:Liouville transform}.	

We aim to estimate $ \tilde{A}_j $ and $ \sin(\tilde{\beta}_j k) $ for $ j=1,2,\cdots,~2^L $. For $\tilde{A}_1$, given that $n$ is prescribed on $[\alpha,1]$, the quantity $M_1$ is known. Using the inequality $x + \frac{1}{x} \geq 2$ for $x > 0$, we obtain the lower bound
\begin{equation}\label{eq:tilde_A_1}
    |\tilde{A}_1|\geq M_1 2^{L-1},\quad 
\end{equation}
where $M_1$ is given in Theorem\,\ref{thm:uniqueness}.

For $\tilde{A}_j$ with $2 \leq j \leq 2^L$, given the bounds $0 < n_* < n < n^*$, we have
\begin{equation*}
	\left(\frac{n_*}{n^*}\right)^{\frac{1}{4}} \leq \frac{\left[n_{l-1}(R_{l-1})\right]^{1/4}}{\left[n_l(R_{l-1})\right]^{1/4}}\leq 	\left(\frac{n^*}{n_*}\right)^{\frac{1}{4}},~~ 2\leq l\leq L.
\end{equation*}
Note that for any $x>0$,
\begin{equation*}
\left|x - \frac{1}{x}\right| < \left|x + \frac{1}{x}\right|.
\end{equation*}
Consequently,
\begin{eqnarray*}
&& \left|\frac{\left[n_{l-1}(R_{l-1})\right]^{1/4}}{\left[n_l(R_{l-1})\right]^{1/4}}
	+ (-1)^{i'_l} \frac{\left[n_l(R_{l-1})\right]^{1/4}}{\left[n_{l-1}(R_{l-1})\right]^{1/4}}\right|
	=\left|\frac{\left[n_{l-1}(R_{l-1})\right]^{1/4}}{\left[n_l(R_{l-1})\right]^{1/4}}
	+ (-1)^{i'_l} \frac{1}{\frac{\left[n_{l-1}(R_{l-1})\right]^{1/4}}{\left[n_l(R_{l-1})\right]^{1/4}}}\right| \nonumber\\
&\leq& \left|\frac{\left[n_{l-1}(R_{l-1})\right]^{1/4}}{\left[n_l(R_{l-1})\right]^{1/4}}+\frac{1}{\frac{\left[n_{l-1}(R_{l-1})\right]^{1/4}}{\left[n_l(R_{l-1})\right]^{1/4}}}\right|
\leq \left(\frac{n_*}{n^*}\right)^{\frac{1}{4}} +\left(\frac{n^*}{n_*}\right)^{\frac{1}{4}}.
\end{eqnarray*}
Inserting this estimate into the expression for $\tilde{A}_j$ yields
\begin{equation}\label{eq:upper bound for tilde_A_j}
   \left|\tilde{A}_j\right|\leq M_2^L,~2\leq j\leq 2^L,
\end{equation}
where $M_2$ is given in Theorem\,\ref{thm:uniqueness}.

Furthermore, the expression given in Lemma~\ref{lem:main_term_d(k)} contains the product  
\begin{equation*}
  \prod_{l=2}^{L} \left( \frac{\left[n_{l-1}(R_{l-1})\right]^{1/4}}{\left[n_l(R_{l-1})\right]^{1/4}}
	+ (-1)^{i'_l} \frac{\left[n_l(R_{l-1})\right]^{1/4}}{\left[n_{l-1}(R_{l-1})\right]^{1/4}} \right).
\end{equation*}
In this product, at least one factor satisfies $i'_l = 1$, meaning at least one term takes the form 
$\left(\frac{\left[n_{l-1}(R_{l-1})\right]^{1/4}}{\left[n_l(R_{l-1})\right]^{1/4}}-\frac{\left[n_l(R_{l-1})\right]^{1/4}}{\left[n_{l-1}(R_{l-1})\right]^{1/4}} \right)$. For such a term we have  
\begin{equation*}
  \begin{aligned}
     \left|\left( \frac{\left[n_{l-1}(R_{l-1})\right]^{1/4}}{\left[n_l(R_{l-1})\right]^{1/4}}-\frac{\left[n_l(R_{l-1})\right]^{1/4}}{\left[n_{l-1}(R_{l-1})\right]^{1/4}} \right)\right|&=\left|\frac{1}{\left[n_l(R_{l-1})\right]^{1/4}\left[n_{l-1}(R_{l-1})\right]^{1/4}}\frac{n_{l-1}(R_{l-1})-n_l(R_{l-1})}{\left[n_{l-1}(R_{l-1})\right]^{1/2}+\left[n_l(R_{l-1})\right]^{1/2}}\right|\\
     &\leq \frac{\varepsilon_0}{2n_{*}}.~
  \end{aligned}
\end{equation*}
Combining this estimate with inequality~\eqref{eq:upper bound for tilde_A_j} yields
\begin{equation}\label{eq:upper bound for tilde_A_j2}
	\left|\tilde{A}_j\right|\leq M_2^{L-1}\frac{\varepsilon_0}{2n_{*}}, ~ 2 \leq j \leq 2^L-1.
\end{equation}

Finally, through direct calculation, we obtain the lower bound
\begin{equation}\label{eq:estimation for tilde A_2^L}
 |\tilde{A}_{2^L}|>2^{L-1}\widetilde{M}_2, 
\end{equation}
where $\widetilde{M}_2$ is given in Theorem\,\ref{thm:uniqueness}.

On the other hand, since
\begin{equation*}
	\begin{aligned}
		\left|\Re\left\{\sin\left(\tilde{\beta}_j k\right)\right\}\right|=\frac{\exp\left(-|\tilde{\beta}_j|\cdot|\Im k|\right)+\exp\left(|\tilde{\beta}_j|\cdot|\Im k|\right) }{2}\left|\sin\left(\tilde{\beta}_j\Re k\right)\right|,
	\end{aligned}
\end{equation*}
it follows that
\begin{equation}\label{eq:bounds for real part of sin tilde beta_j k}
   \frac{1}{2}\exp\left(|\tilde{\beta}_j|\cdot|\Im k|\right) \left|\sin\left(\tilde{\beta}_j\Re k\right)\right| \leq \left|\Re\left\{\sin\left(\tilde{\beta}_j k\right)\right\}\right|\leq \exp\left(|\tilde{\beta}_j|\cdot|\Im k|\right) \left|\sin\left(\tilde{\beta}_j\Re k\right)\right|.
\end{equation}
Consequently, we obtain the following estimates:
\begin{equation}\label{eq:upper and lower estimation for sin}
  \left\{
  \begin{aligned}
   &\left|\Re \left(\sin\left(\tilde{\beta}_jk\right)\right)\right|\leq \exp{\left(\left|\tilde{\beta}_j\right|\cdot\left|\Im k\right|\right)}\left|\sin\left(\tilde{\beta}_j\Re k\right)\right|,~\\
   &\frac{1}{2}\left| \sin\left(\tilde{\beta}_j \Re k\right) \right|\exp{\left(|\tilde{\beta}_j| \cdot |\Im k|\right)}\leq \left| \Re\left(\sin\left(\tilde{\beta}_jk\right)\right) \right|\leq \left|\sin\left(\tilde{\beta}_jk\right)\right|\leq \exp{\left(\left|\tilde{\beta}_j\right|\cdot \left|\Im k\right|\right)}.~
  \end{aligned}
  \right.
\end{equation}

Let $\beta = \min_{2 \leq j \leq 2^L} \bigl\{ \tilde{\beta}_1 - |\tilde{\beta}_j| \bigr\}$. 
From Weyl's law and the Cartwright--Levinson theorem \cite{CCH16,LV15}, we obtain for some 
constant $C_1 > 0$ the estimate
\begin{equation}\label{eq:R(k)}
	|R(d(k))| \leq C_1 \frac{\exp\{\tilde{\beta}_1|\Im k|\}}{k^2}.
\end{equation}

Case $1$. $\tilde{\beta}_1$ and $\tilde{\beta}_{2^L}$ are rationally linearly independent, and $\varepsilon_0\leq \frac{n_{*}M_1}{24M_2^{L-1}}$. We choose  
\begin{equation}\label{eq:Kronecker theorem linear independence set up}
  \begin{aligned}
   \alpha_1=\frac{1}{4},~ \alpha_{2^L}=0,~ \varepsilon_1=\min\left\{\frac{1}{12}, \frac{M_1}{96\pi M_2^L}\right\},~T_1=\frac{24C_1}{M_1}+1,~
  \end{aligned}
\end{equation}
where $C_1$ is the constant defined in \eqref{eq:R(k)}.

Step $1$. By Lemma~\ref{lem:Kronecker_approximation}, there exist $t_1 > T_1$ and integers $p_1,~p_{2^L}$ such that
\begin{equation}\label{eq:linearly independent Kronecker estimation}
	\begin{aligned}
		\left|\tilde{\beta}_jt_1-2\pi p_j-2\pi\alpha_j\right|<2\pi \varepsilon_1,~ \text{for}~j=1,2^L.
	\end{aligned}
\end{equation}

Let
\begin{equation}\label{eq:C2N1}
    C_2 = \max\left\{ \frac{1}{\beta} \ln\left( \frac{24M_2^L}{M_1} \right), T_1 \right\}, \quad
N_1=\left[\frac{\tilde{\beta}_1}{\pi}\sqrt{t_1^2+C_2^2}-\frac{1}{2}\right]+1.
\end{equation}
Then the circle $\left|k\right|=\left(\frac{N_1+\frac{1}{2}}{\tilde{\beta}_1}\pi\right)$ intersects the line $\Re k=-t_1$ at the points $E_1=\left(-t_1,-\sqrt{\left(\frac{N_1+\frac{1}{2}}{\tilde{\beta}_1}\pi\right)^2-t_1^2}\right)$ and $F_1=\left(-t_1,\sqrt{\left(\frac{N_1+\frac{1}{2}}{\tilde{\beta}_1}\pi\right)^2-t_1^2}\right)$,
and intersects the line $\Re k=t_1$ at $E_2=\left(t_1,-\sqrt{\left(\frac{N_1+\frac{1}{2}}{\tilde{\beta}_1}\pi\right)^2-t_1^2}\right)$ and $F_2=\left(t_1,\sqrt{\left(\frac{N_1+\frac{1}{2}}{\tilde{\beta}_1}\pi\right)^2-t_1^2}\right)$.

Then the simple closed curve $\gamma_1$ is composed of the circular arcs $\stackrel{\frown}{E_1E_2}$ and $\stackrel{\frown}{F_1F_2}$,~together with the line segments $\overline{E_1F_1}$ and $\overline{E_2F_2}$,~as shown in Figure 2.
\begin{figure}[htbp]\label{fig:gamma1}
  \centering
  \includegraphics[width=0.8\textwidth]{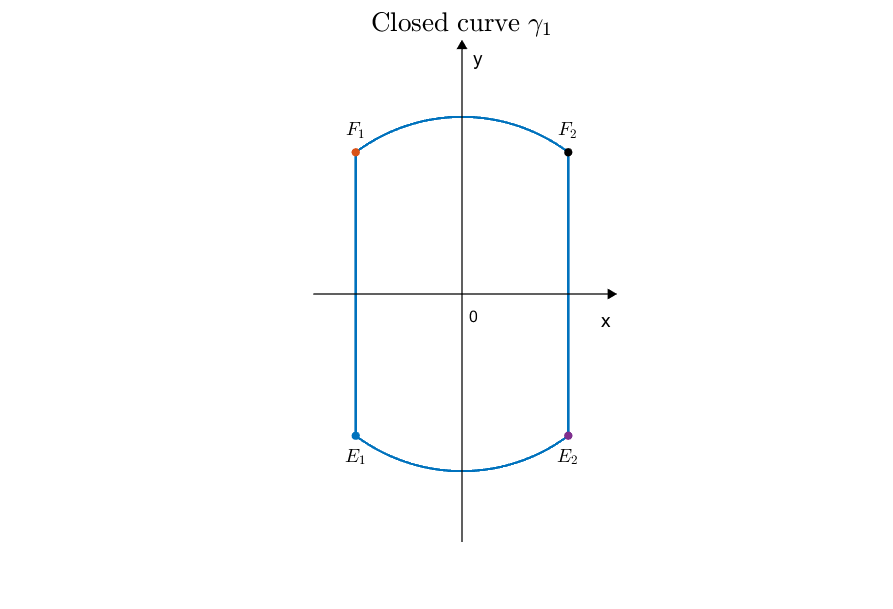}
  \caption{The figure of closed curve $\gamma_1$.}
\end{figure}

Step 2. First, from \eqref{eq:def f g} and Lemma~\ref{lem:y,y'}, we obtain that for any $k \in \mathbb{C}$ with $|k| \geq 1$,
\begin{eqnarray}\label{eq:estimation_f(k)}
	|f(k)| & \leq & \left| k^2 \int_0^{\alpha} \left[n^{(2)} - n^{(1)}\right] y_1(r) y_2(r) \, \mathrm{d} r \right| \nonumber \\
		& \leq & |k|^2 \int_0^{\alpha} |n^{(2)} - n^{(1)}| |y_1(r) y_2(r)| \, \mathrm{d} r \nonumber \\
		& \leq & C |k|^2  \left(\frac{1}{|k|} + \frac{1}{|k|^2} \right)^2 \exp \left( 2|\Im k | \int_{0}^{\alpha} \sqrt{n(s)} \, \mathrm{d} s\right)  \int_0^{1} |n^{(2)} - n^{(1)}| \, \mathrm{d} r  \nonumber \\
		& \leq & C_3 |k|^2  \frac{\exp\left(|\Im k| \left(1 + \delta_L - \varepsilon\right)\right)}{|k|^2}   \nonumber \\
		& \leq & C_3 \exp\left(|\Im k | \left(1 + \delta_L - \varepsilon\right)\right),
\end{eqnarray}
where $C_3$ is a constant that depends only on $n^{(1)}$ and $n^{(2)}$ and is independent of $k$.

Second, we estimate $|d(k)|$ on the circular arcs $\stackrel{\frown}{E_1E_2}$ and $\stackrel{\frown}{F_1F_2}$. Let $d(k) = D(d(k)) + R(d(k))$, where $R(d(k))$ denotes the remainder term. From Lemma~\ref{lem:main_term_d(k)} we obtain
\begin{equation}\label{eq:expression of d(k) on arc independent}
	\begin{aligned}
		\left| d(k) \right|=\left| D(d(k))+R(d(k)) \right|=\left| \frac{1}{2^L k}\sum_{j=1}^{2^L}\tilde{A}_j \sin{\left(\tilde{\beta}_jk\right)}+R(d(k))\right|.
	\end{aligned}
\end{equation}

The following inequality holds (\cite{MP94}):
\begin{equation}\label{eq:estimation of sin on arc}
	\left|\sin\left\{\tilde{\beta}_1k\right\}\right| \geq \frac{1}{4} \exp\left\{|\Im k|\tilde{\beta}_1\right\},
\end{equation}
which is valid for \( |k| = \frac{(N + \frac{1}{2})\pi}{\tilde{\beta}_1} \) with \( N \in \mathbb{Z} \), and consequently on the arcs \(\stackrel{\frown}{E_1E_2}\) and \(\stackrel{\frown}{F_1 F_2}\).

Combining \eqref{eq:upper bound for tilde_A_j}, \eqref{eq:upper and lower estimation for sin}, the choice of $T_1$ in \eqref{eq:Kronecker theorem linear independence set up}, and the choices of $C_2$ and $N_1$ from \eqref{eq:C2N1}, we have
\begin{equation}\label{eq:estimation of part of main term independent case}
  \begin{aligned}
   &\left| \sum_{j=2}^{2^L}\frac{\tilde{A}_j}{2^L k}\sin{\left(\tilde{\beta}_j k\right)} \right|\leq\sum_{j=2}^{2^L}\frac{M_2^L}{2^L |k|}\exp{\left(\left|\tilde{\beta}_j\right||\Im k|\right)}\leq \frac{M_2^L}{|k|}\exp{\left(\max_{2\leq j\leq 2^L}\left|\tilde{\beta}_j\right||\Im k|\right)}<\frac{M_1}{24|k|}\exp{\left(\tilde{\beta}_1 |\Im k|\right)}.~
  \end{aligned}
\end{equation}

On the arcs $\overset{\frown}{E_1E_2}$ and $\overset{\frown}{F_1F_2}$, the definitions of $C_2$ and $N_1$ in \eqref{eq:C2N1} yield
\begin{eqnarray}\label{eq:k and C0 and Im k on arc}
|k|&=&\frac{N_1+\frac{1}{2}}{\tilde{\beta}_1}\pi>\sqrt{t_1^2+C_2^2}>C_2,  \quad |\Im k|\geq \left| \sqrt{\left(\frac{N_1+\frac{1}{2}}{\tilde{\beta}_1}\pi\right)^2-t_1^2}\right|>C_2,    \nonumber\\
C_2 &\geq & \frac{24C_1}{M_1}, \quad C_2\geq \frac{1}{\beta}\ln{\frac{24M_2^L}{M_1}},
\end{eqnarray}
where $C_2$ is the constant defined in \eqref{eq:C2N1}.

On the other hand, combining \eqref{eq:R(k)} with \eqref{eq:k and C0 and Im k on arc} gives
\begin{equation}\label{eq:arc independent parts of first term kill R(d(k))}
	|R(d(k))| \leq \frac{C_1}{|k|}\exp{\left\{\left|\Im k\right|\tilde{\beta}_1\right\}}\frac{1}{|k|}\leq  \frac{M_1}{24|k|}\exp{\left\{\left|\Im k\right|\tilde{\beta}_1\right\}}.
\end{equation}

Substituting \eqref{eq:tilde_A_1}, \eqref{eq:estimation of sin on arc}, \eqref{eq:estimation of part of main term independent case}, and \eqref{eq:arc independent parts of first term kill R(d(k))} into \eqref{eq:expression of d(k) on arc independent} and applying the triangle inequality yields
\begin{eqnarray}\label{eq:tentative estimation of d(k) on arc independent}
		\left| d(k) \right|& \geq & \left| \frac{1}{2^L |k|} \frac{\left|\tilde{A}_1\right|}{4}\exp{\left\{\left|\Im k\right|\tilde{\beta}_1\right\}}-\left|\frac{1}{2^L \left|k\right|}\sum_{j=2}^{2^L}\left|\tilde{A}_j\right|\exp{\left\{\left|\Im k\right| \left|\tilde{\beta}_j\right|\right\}}\right|-\left| R(d(k)) \right|\right| \nonumber \\
		& \geq & \left| \frac{M_1}{8\left|k\right|}\exp{\left\{\left|\Im k\right|\tilde{\beta}_1\right\}}-\frac{M_2^L}{2^L |k|}\sum_{j=2}^{2^L}\exp{\left\{\left|\Im k\right| \max_{2\leq j\leq 2^L}{\left|\tilde{\beta}_j\right|}\right\}}-\frac{C_1}{\left|k\right|^2}\exp{\left\{\left|\Im k\right| \tilde{\beta}_1\right\}} \right| \nonumber\\
&\geq &\frac{M_1}{24|k|}\exp{\left\{\left|\Im k\right| \tilde{\beta}_1\right\}}.
\end{eqnarray}
This gives the lower bound for $|d(k)|$ on the arcs $\overset{\frown}{E_1E_2}$ and $\overset{\frown}{F_1F_2}$.

Third, we estimate $|d(k)|$ from below on the line segments $\overline{E_1F_1}$ and $\overline{E_2F_2}$. From \eqref{eq:linearly independent Kronecker estimation} we obtain
\begin{equation*}
 \sin{\left(\tilde{\beta}_1t_1\right)}>\frac{1}{2},\quad
 \left|\sin{\left(\tilde{\beta}_{2^L}t_1\right)}\right|<\sin{\left(2\pi\varepsilon_1\right)}<\frac{M_1}{24M_2^L}.
\end{equation*}

Combining \eqref{eq:upper bound for tilde_A_j},~\eqref{eq:upper bound for tilde_A_j2},~ \eqref{eq:upper and lower estimation for sin}, the choice of $\varepsilon_1$ and $\alpha_{2^L}$ in \eqref{eq:Kronecker theorem linear independence set up}, and the inequality \eqref{eq:linearly independent Kronecker estimation}, we have
\begin{equation}\label{eq:estimation of parts of main term of dk on line}
  \begin{aligned}
    \left| \Re\left(\sum_{j=2}^{2^L}\frac{\tilde{A}_j}{2^L}\sin{\left(\tilde{\beta}_j k\right)}\right) \right| &\leq \left| \sum_{j=2}^{2^L-1}\frac{M_2^{L-1}\varepsilon_0}{2n_{*}}\frac{\exp{\left(\left|\tilde{\beta}_j\right||\Im k|\right)} }{2^L}\left|\sin{\left(\tilde{\beta}_j\Re k\right)}\right|+M_2^L\exp{\left(\left|\tilde{\beta}_{2^L}\right| |\Im k|\right)}\frac{\sin{\left(\tilde{\beta}_{2^L}\Re k\right)} }{2^L} \right|\\
    &\leq \left| \sum_{j=2}^{2^L-1}\frac{M_2^{L-1}\varepsilon_0}{2n_{*}}\frac{\exp{\left(\left|\tilde{\beta}_j\right||\Im k|\right)} }{2^L}\left|\sin{\left(\tilde{\beta}_jt_1\right)}\right|+M_2^L\exp{\left(\left|\tilde{\beta}_{2^L}\right| |\Im k|\right)}\frac{\left|\sin{\left(\tilde{\beta}_{2^L}t_1\right)}\right| }{2^L} \right|\\
    &\leq M_2^{L-1}\frac{\varepsilon_0}{2n_{*}}\exp{\left( \tilde{\beta}_1 |\Im k|\right)}+M_2^L\frac{\exp{\left( \tilde{\beta}_1 |\Im k|\right)}}{2^L}\sin{\left(2\pi \varepsilon_1\right)}<\frac{M_1}{24}\exp{\left( \tilde{\beta}_1 |\Im k|\right)}
  \end{aligned}
\end{equation}

Moreover, from the choice of $\varepsilon_1$ and $\alpha_1$ in \eqref{eq:Kronecker theorem linear independence set up} together with inequalities \eqref{eq:tilde_A_1},~\eqref{eq:upper and lower estimation for sin} and \eqref{eq:linearly independent Kronecker estimation}, we obtain
\begin{equation}\label{eq:estimation of first term in main term on line independent case}
  \begin{aligned}
    \left| \frac{\tilde{A}_1}{2^L}\sin{\left(\tilde{\beta}_1 k\right)} \right|\geq \left| \Re\left(\frac{\tilde{A}_1}{2^L}\sin{\left(\tilde{\beta}_1 k\right)}\right) \right|\geq \left| \frac{M_1 2^{L-1}}{2^L}\frac{1}{2}\exp{\left(\tilde{\beta}_1|\Im k|\right)}\sin{\left( \tilde{\beta}_1\Re k \right)} \right|\geq \frac{M_1}{8}\exp{\left(\tilde{\beta}_1|\Im k|\right)}.~
  \end{aligned}
\end{equation}
Furthermore, using the choice of $T_1$ in~\eqref{eq:Kronecker theorem linear independence set up}, 
$C_2$ in~\eqref{eq:C2N1}, and inequality~\eqref{eq:R(k)}, we obtain on the line segments 
$\overline{E_1F_1}$ and $\overline{E_2F_2}$ the estimate
\begin{equation*}
\left|R(d(k))\right| \leq \frac{M_1}{24|k|}\exp{\left\{\left|\Im k\right|\tilde{\beta}_1\right\}},
\end{equation*}
which coincides with the bound in~\eqref{eq:arc independent parts of first term kill R(d(k))}.

Substituting \eqref{eq:estimation of parts of main term of dk on line},~\eqref{eq:estimation of first term in main term on line independent case} and \eqref{eq:arc independent parts of first term kill R(d(k))} into \eqref{eq:expression of d(k) on arc independent}, we have
\begin{eqnarray}\label{eq:estimation of d(k) on line independent}
    |d(k)|&=&\left| D(d(k))+R(d(k)) \right|=\frac{1}{|k|}\left| kD(d(k))+kR(d(k)) \right|\geq \frac{1}{|k|}\left| \left| \Re\left(kD(d(k))\right) \right|-\left|kR(d(k))\right| \right| \nonumber\\
    &\geq& \frac{1}{|k|}\left| \left|\Re \frac{\tilde{A}_1}{2^L}\sin\left(\tilde{\beta}_1k\right)\right|-\left| \Re \sum_{j=2}^{2^L}\frac{\tilde{A}_j}{2^L}\sin\left(\tilde{\beta}_jk\right) \right|-\left|kR(d(k))\right| \right| \nonumber\\
    &\geq& \left| \frac{M_1}{8}\exp{\left(\tilde{\beta}_1\left|\Im k\right|\right)}-\frac{M_1}{24}\exp{\left(\tilde{\beta}_1\left|\Im k\right|\right)}-\frac{M_1}{24}\exp{\left(\tilde{\beta}_1\left|\Im k\right|\right)} \right| \nonumber\\
    &\geq& \frac{M_1}{24}\exp{\left(\tilde{\beta}_1\left|\Im k\right|\right)}
  \end{eqnarray}
This gives the lower bound for $|d(k)|$ on $\overline{E_1F_1}$ and $\overline{E_2F_2}$.

Finally, combining \eqref{eq:estimation_f(k)} with the lower bounds \eqref{eq:tentative estimation of d(k) on arc independent} and \eqref{eq:estimation of d(k) on line independent} for $|d(k)|$ on $\gamma_1$, we obtain the estimate
	\begin{equation}\label{eq:estimation of g(k) on gamma1}
			\left|g(k)\right|=\left|\frac{f(k)}{d(k)}\right| \leq \frac{24C_3|k|}{M_1\exp{\left\{\left|\Im k \right|\varepsilon\right\}}}, \quad k\in \gamma_1,
	\end{equation}
where the constant $C_3$ depends only on $n^{(1)}$ and $n^{(2)}$.

Let $G_1$ be the region bounded by $\gamma_1$. Since $f(k)$ and $d(k)$ are analytic and have the same zeros (\cite{CCH16, MP94}), the function $g(k)$ is analytic in the complex plane. In particular, $g(k)$ is continuous on $\overline{G}_1$, and we obtain the estimate
\begin{equation}\label{eq:estimation of |g(k)| on the region bounded by gamma1}
  |g(k)| \leq C (|k|+1), \quad k\in G_1.
\end{equation}

	
Step $3$. Let $T_2 = 2T_1$, keeping $\alpha_1$, $\alpha_{2^L}$, $\varepsilon_1$, and $C_2$ as chosen in Step~1. By Lemma~\ref{lem:Kronecker_approximation}, we can choose $t_2 > \max\{t_1, T_2\}$ such that there exist integers $p^{(2)}_1$, $p^{(2)}_{2^L}$ satisfying
\begin{equation*}
    \left|\tilde{\beta}_j t_2 - 2\pi p^{(2)}_j - 2\pi\alpha_j\right| < 2\pi \varepsilon_1,\quad \text{for } j=1,2^L.
\end{equation*}

Now set $N_2=\left[\frac{\tilde{\beta}_1}{\pi}\sqrt{t_2^2+C_2^2}-\frac{1}{2}\right]+1$. The circle $\left|k\right|=\frac{N_2+\frac{1}{2}}{\tilde{\beta}_1}\pi$ meets the line $\Re k=-t_2$ at the points ${E}_{21}=\left(-t_2,-\sqrt{\left(\frac{N_2+\frac{1}{2}}{\tilde{\beta}_1}\pi\right)^2-t_2^2}\right)$ and ${F}_{21}=\left(-t_2,\sqrt{\left(\frac{N_2+\frac{1}{2}}{\tilde{\beta}_1}\pi\right)^2-t_2^2}\right)$,
and meets the line $\Re k=t_2$ at ${E}_{22}=\left(t_2,-\sqrt{\left(\frac{N_2+\frac{1}{2}}{\tilde{\beta}_1}\pi\right)^2-t_2^2}\right)$ and ${F}_{22}=\left(t_2,\sqrt{\left(\frac{N_2+\frac{1}{2}}{\tilde{\beta}_1}\pi\right)^2-t_2^2}\right)$.
The simple closed curve $\gamma_2$ is then formed by the circular arcs $\stackrel{\frown}{{E}_{21}{E}_{22}}$ and $\stackrel{\frown}{{F}_{21} {F}_{22}}$, together with the line segments $\overline{{E}_{21}{F}_{21}}$ and $\overline{{E}_{22}{F}_{22}}$.

Following the same argument as in Step~2, we obtain the estimate for $|g(k)|$ on $\gamma_2$:
\begin{equation*}
  \left|g(k)\right|<\frac{24{C}_3|k|}{M_1\exp{\left\{\left|\Im k\right|\varepsilon\right\} }}, \quad k\in \gamma_2 .
\end{equation*}

\begin{figure}[htbp]\label{fig:gamma2}
	\centering
	\includegraphics[width=0.6\textwidth]{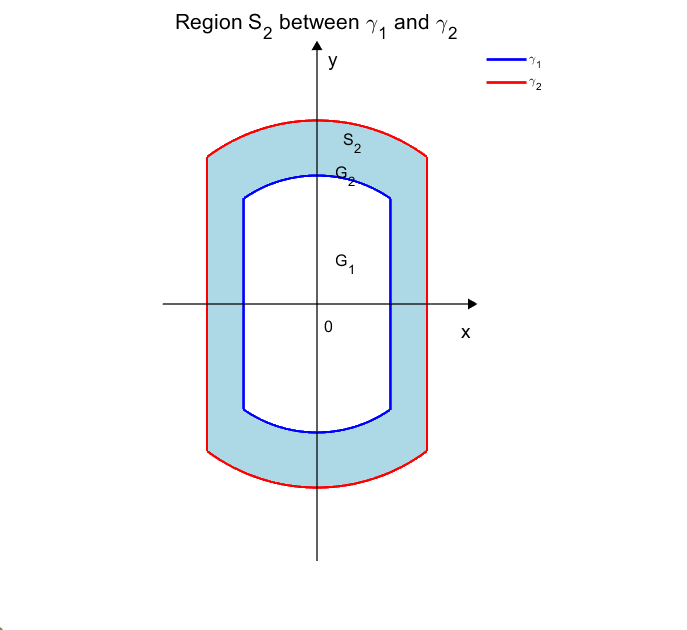}
	\caption{The figure of $S_2$ bounded by $\gamma_1$ and $\gamma_2$.}
\end{figure}
Let ${S}_2$ be the region bounded by $\gamma_1$ and $\gamma_2$, and define $G_2 = G_1 \cup S_2$, as shown in Figure~3. Since $g(k)$ is analytic on $S_2$ and $k$ has no zeros in $S_2$, the function $g(k)/k$ is analytic on the closure of $S_2$. By the maximum modulus principle (\cite{SS10}), the maximum of $|g(k)/k|$ on $S_2$ is attained on the boundary $\gamma_1 \cup \gamma_2$. Combining this result with the estimate \eqref{eq:estimation of |g(k)| on the region bounded by gamma1} on $G_1$, we obtain
\begin{equation}\label{eq:estimation of |g(k)| on region bounded by gamma2}
 |g(k)| \leq C \leq  C(|k|+1), \quad k\in G_2.
\end{equation}

Similarly, for each $s \geq 2$, define $T_s = 2^{s-1}T_1$. Then the estimate \eqref{eq:estimation of |g(k)| on region bounded by gamma2} holds on the closure of the region $G_{s}$ (bounded by $\gamma_s$). Since $G_{s}$ exhaust the complex plane as $s \to \infty$, we conclude that \eqref{eq:estimation of |g(k)| on region bounded by gamma2} is valid on all of $\mathbb{C}$.

	
Step $4$. Since $g(k)$ is an entire function, it admits a Taylor expansion (\cite{SS10}):
\begin{equation*}
	g(k)=\sum_{j=0}^{\infty}\frac{g^{(j)}(0)}{j!}k^j,
\end{equation*}
where $k \in \mathbb{C}$, and $g^{(j)}(0)$ denotes the value of the $j$-th derivative of $g$ at $0$. Furthermore, by the Cauchy integral formula \cite{SS10}, we have
\begin{equation}\label{eq:Cauchy integral formula for gj(0)}
	g^{(j)}(0)=\int_{|\xi|=R}\frac{g(\xi)}{\xi^{j+1}}\mathrm{d}\xi.
\end{equation}
Substituting the estimate \eqref{eq:estimation of |g(k)| on region bounded by gamma2} into \eqref{eq:Cauchy integral formula for gj(0)} yields	
\begin{equation*}
	|g^{(j)}(0)|\leq \int_{|\xi|=R}\left|\frac{g(\xi)}{\xi^{j+1}}\right|\mathrm{d}\xi\leq \int_{|\xi|=R}\frac{C(|\xi|+1)}{R^{j+1}}\mathrm{d}\xi\leq \frac{C\pi}{R^{j-1}},
\end{equation*}
which implies that $|g^{(j)}(0)| \to 0$ as $R \to \infty$ for all $j \geq 2$. Therefore, $g(k)$ must be a polynomial of degree at most one.
	
Next, take $v_j = \mathrm{i} N_j$, where $N_j$ is defined by induction. It follows from \eqref{eq:estimation of g(k) on gamma1} that
\begin{equation*}
    \left|g(v_j)\right|\leq \frac{24{C}_3|\mathrm{i}N_j|}{M_1\exp{\left\{\left|\Im \left(\mathrm{i}N_j\right) \right|\varepsilon\right\}}}=\frac{24{C}_3|N_j|}{M_1\exp{\left\{\left|N_j \right|\varepsilon\right\}}} \to 0, \quad j\to \infty,
\end{equation*}
which implies that $g(k) \equiv 0$. Together with \eqref{eq:def f g}, it follows that $f(k) \equiv 0$ for all $k \in \mathbb{C}$. In particular, $f(\mathrm{i}s) = 0$ for every $s \in \mathbb{R}$. In \eqref{eq:y_1,y_2a}, the equation we study is
$y_{j}'' + k^2 y_{j} = 0$,~$j=1,2$.
Substituting $k = \mathrm{i}s,\ s \in \mathbb{R}$ into (3.10), we obtain the form of the equation appearing in Lemma~\ref{lem:completeness}, then using Lemma~\ref{lem:completeness}, we therefore conclude that $n^{(1)} \equiv n^{(2)}$.\\

	
Case $2$. $\tilde{\beta}_1,~\tilde{\beta}_{2^L}$ are rational linear dependent. We have  
\begin{equation}\label{eq:L layer linear dependence relation}
  \tilde{\beta}_{2^L}=1-\delta_L=\frac{\hat{q}}{\hat{p}}\left(1+\delta_L\right)=\frac{\hat{q}}{\hat{p}}\tilde{\beta}_1,~
\end{equation}
where $0<|\hat{q}|<\hat{p}$.

(a) $\delta_L>1$, $n_L(1)>1$, and $\varepsilon_0\leq \frac{M_1\widetilde{M}_1n_{*}}{12M_2^{L-1}}$. The expression for $\tilde{A}_j$ in Lemma~\ref{lem:main_term_d(k)} gives
\begin{equation*}\label{eq:L layer subcase  first set up}
    \tilde{\beta}_{2^L}=1-\delta_L<0,~0<-\hat{q}<\hat{p},~\tilde{A}_1>0.~
\end{equation*}

Step $1$. Choose  
\begin{eqnarray}\label{eq:L layer linear dependence Kronecker set up}
     \alpha_1&=&1+\frac{1}{\hat{p}+|\hat{q}|}, \quad \varepsilon_1=\min\left\{\frac{1}{2(\hat{p}+|\hat{q}|)},~\frac{4|\hat{q}|(\hat{p}-|\hat{q}|)}{\hat{p}+|\hat{q}|},~\frac{4\hat{p}(\hat{p}-|\hat{q}|)}{\hat{p}+|\hat{q}|}, \frac{|\hat{q}|}{2\hat{p}(\hat{p}+|\hat{q}|)}\right\}, \nonumber~\\
     \widetilde{M}_1&=&\min\left\{ \left|\sin\left(\frac{\hat{p}-|\hat{q}|}{\hat{p}+|\hat{q}|}\pi+2\hat{p}\varepsilon_1\right)\right|,~\left|\sin\left(\frac{\hat{p}-|\hat{q}|}{\hat{p}+|\hat{q}|}\pi-2\hat{p}\varepsilon_1\right)\right|\right\},\\
     \widetilde{T}_1&=&\max\left\{\frac{24C_1}{M_1}+1,~\frac{4C_1}{M_1\widetilde{M}_1}+1\right\}. \nonumber
\end{eqnarray}
From Lemma~\ref{lem:Kronecker_approximation} and \eqref{eq:L layer linear dependence relation}, 
there exist $t_1 > \widetilde{T}_1$ and integers $p_1$ such that
\begin{equation}\label{eq:L layer linear dependence subcase a Kronecker inequality}
\left| \tilde{\beta}_1t_1-2\pi\alpha_1\hat{p}-2\pi p_1\hat{p} \right|\leq 2\pi\varepsilon_1\hat{p},\quad
\left| \tilde{\beta}_{2^L}t_1-2\pi\alpha_1\hat{q}-2\pi p_1\hat{q} \right|\leq 2\pi\varepsilon_1|\hat{q}|.
\end{equation}

Let $\widetilde{N}_1=\left[\frac{\tilde{\beta}_1}{\pi}\sqrt{t_1^2+C_2^2}-\frac{1}{2}\right]+1$, where $C_2=\max\left\{\frac{1}{\beta}\ln\left(\frac{24M_2^L}{M_1}\right)+1,\widetilde{T}_1\right\}$. The circle $|k|=\frac{\widetilde{N}_1+\frac{1}{2}}{\tilde{\beta}_1}\pi$ meets the line $\Re k=-t_1$ at the points  
$\widetilde{E}_1=\left(-t_1,-\sqrt{\left(\frac{\widetilde{N}_1+\frac{1}{2}}{\tilde{\beta}_1}\pi\right)^2-t_1^2}\right)$ and  
$\widetilde{F}_1=\left(-t_1,\sqrt{\left(\frac{\widetilde{N}_1+\frac{1}{2}}{\tilde{\beta}_1}\pi\right)^2-t_1^2}\right)$,  
and meets the line $\Re k=t_1$ at  
$\widetilde{E}_2=\left(t_1,-\sqrt{\left(\frac{\widetilde{N}_1+\frac{1}{2}}{\tilde{\beta}_1}\pi\right)^2-t_1^2}\right)$ and  
$\widetilde{F}_2=\left(t_1,\sqrt{\left(\frac{\widetilde{N}_1+\frac{1}{2}}{\tilde{\beta}_1}\pi\right)^2-t_1^2}\right)$.  
The simple closed curve $\widetilde{\gamma}_1$ is then formed by the circular arcs $\overset{\frown}{\widetilde{E}_1\widetilde{E}_2}$ and $\overset{\frown}{\widetilde{F}_1\widetilde{F}_2}$, together with the line segments $\overline{\widetilde{E}_1\widetilde{F}_1}$ and $\overline{\widetilde{E}_2\widetilde{F}_2}$.

Step 2. On the circular arcs $\overset{\frown}{\widetilde{E}_1\widetilde{E}_2}$ and $\overset{\frown}{\widetilde{F}_1\widetilde{F}_2}$, the same derivation as in Case~1 yields an estimate identical to~\eqref{eq:tentative estimation of d(k) on arc independent}. We now turn to estimating $|d(k)|$ on the line segments $\overline{\widetilde{E}_1\widetilde{F}_1}$ and $\overline{\widetilde{E}_2\widetilde{F}_2}$.
     
From the expression of $\tilde{A}_j$ in Lemma~\ref{lem:main_term_d(k)}, together with~\eqref{eq:L layer linear dependence Kronecker set up} and~\eqref{eq:L layer linear dependence subcase a Kronecker inequality}, we derive
\begin{equation}\label{eq:L layer linear dependence subcase a same sign}
    \begin{aligned}
      \left(\tilde{A}_1\sin\left(\tilde{\beta}_1t_1\right)\right)\cdot\left(\tilde{A}_{2^L}\sin\left(\tilde{\beta}_{2^L}t_1\right)\right)>0.~
    \end{aligned}
\end{equation}
Combining \eqref{eq:upper bound for tilde_A_j2}, \eqref{eq:upper and lower estimation for sin}, and the choice of $\varepsilon_0$, we obtain
\begin{eqnarray}\label{eq:L layer estimation for linear dependence subcase a part of main term}
      &&\left|\Re\left(\sum_{j=2}^{2^L-1}\frac{\tilde{A}_j}{2^L}\sin\left(\tilde{\beta}_j k\right)\right)\right| \leq\left| \sum_{j=2}^{2^L-1}\frac{\tilde{A}_j}{2^L}\exp{\left(\left|\tilde{\beta}_j\right||\Im k|\right)}\sin{\left(\tilde{\beta}_j \Re k\right)} \right| \nonumber\\
      &\leq& \sum_{j=2}^{2^L-1}\frac{M_2^{L-1}\varepsilon_0}{2n_{*}}\frac{1}{2^L}\exp{\left(\left|\tilde{\beta}_j\right|\cdot|\Im k|\right)} \leq \frac{M_1\widetilde{M}_1}{12}\exp{\left(\tilde{\beta}_1|\Im k|\right)}.~
\end{eqnarray}  
Based on the choice of $\widetilde{T}_1$, $C_2$, and $\widetilde{N}_1$, we have
\begin{equation}\label{eq:L layer linear dependence subcase a remainder estimation}
    \begin{aligned}
      \left|R(d(k))\right|\leq \frac{C_1\exp{\left(\tilde{\beta}_1|\Im k|\right)}}{|k|^2}<\frac{M_1\widetilde{M}_1}{12|k|}\exp{\left(\tilde{\beta}_1|\Im k|\right)}.~
    \end{aligned}
\end{equation}
It follows from \eqref{eq:tilde_A_1},  \eqref{eq:upper and lower estimation for sin}, \eqref{eq:L layer linear dependence subcase a Kronecker inequality},  \eqref{eq:L layer linear dependence subcase a same sign}, \eqref{eq:L layer estimation for linear dependence subcase a part of main term}, and \eqref{eq:L layer linear dependence subcase a remainder estimation} that
\begin{eqnarray}\label{eq:L layer dependence subcase a dk estimation on line}
      |d(k)|&=&\frac{1}{|k|}\left|kd(k)\right|=\frac{1}{|k|}\left|kD(d(k))+kR(d(k))\right|\geq \frac{1}{|k|}\left| \left|\Re\left(kD(d(k))\right)\right|-\left|kR(d(k))\right|\right| \nonumber\\
      &=&\frac{1}{|k|}\left| \left|\Re\left(\frac{\tilde{A}_1}{2^L}\sin\left(\tilde{\beta}_1k\right)+\frac{\tilde{A}_{2^L}}{2^L}\sin\left(\tilde{\beta}_{2^L}k\right)+\sum_{j=2}^{2^L-1}\frac{\tilde{A}_j}{2^L}\sin\left(\tilde{\beta}_jk\right)\right)\right|-\left|kR(d(k))\right| \right| \nonumber \\
      &\geq&\frac{1}{|k|}\left| \left|\Re\left(\frac{\tilde{A}_1}{2^L}\sin\left(\tilde{\beta}_1k\right)+\frac{\tilde{A}_{2^L}}{2^L}\sin\left(\tilde{\beta}_{2^L}k\right)\right)\right|-\left|\Re\left(\sum_{j=2}^{2^L-1}\frac{\tilde{A}_j}{2^L}\sin\left(\tilde{\beta}_jk\right)\right)\right|-|kR(d(k))| \right| \nonumber \\
      &\geq& \frac{1}{|k|}\left| \left|\frac{\tilde{A}_1}{2^L}\sin\left(\tilde{\beta}_1t_1\right)\frac{1}{2}\exp{\left(\tilde{\beta}_1|\Im k|\right)}+\frac{\tilde{A}_{2^L}}{2^L}\sin\left(\tilde{\beta}_{2^L}t_1\right)\frac{1}{2}\exp{\left(\tilde{\beta}_{2^L}|\Im k|\right)}\right|\right. \nonumber \\
      &-&\left.\frac{M_1\widetilde{M}_1}{12}\exp{\left(\tilde{\beta}_1|\Im k|\right)}-\frac{M_1\widetilde{M}_1}{12}\exp{\left(\tilde{\beta}_1|\Im k|\right)} \right| \nonumber \\
      &\geq&\frac{M_1\widetilde{M}_1}{12|k|}\exp{\left(\tilde{\beta}_1|\Im k|\right)}.~
\end{eqnarray}
Combining \eqref{eq:tentative estimation of d(k) on arc independent} with \eqref{eq:L layer dependence subcase a dk estimation on line}, we obtain
\begin{equation*}
    |d(k)| \geq \frac{\widetilde{C}}{|k|}\exp{\left(\tilde{\beta}_1|\Im k|\right)},\quad \widetilde{C}=\min\left\{\frac{M_1\widetilde{M}_1}{12},~\frac{M_1}{24}\right\},~
\end{equation*}
for $k\in \tilde{\gamma}_1$. The rest of the proof proceeds similarly to Steps~3 through~5 in Case~1.
      
(b) $\delta_L<1$, $n_L(1)>1$, and $\varepsilon_0\leq \frac{n_{*}M_1}{24M_2^{L-1}}$. The expressions for $\tilde{A}_j$ and $\tilde{\beta}_j$ in Lemma~\ref{lem:main_term_d(k)} yield  
\begin{equation}\label{eq:L layer dependence subcase b first set up}
    \begin{aligned}
        \tilde{A}_1>0,\quad \tilde{A}_{2^L}>0,\quad \tilde{\beta}_1>0,\quad \tilde{\beta}_{2^L}>0,\quad 0<\hat{q}<\hat{p}.~
    \end{aligned}
\end{equation}

Step $1$. Choose  
\begin{equation}\label{eq:L layer dependence contour construction set up}
    \begin{aligned}
        \alpha_1=\frac{1}{4\hat{p}},~\varepsilon_1=\min\left\{\frac{1}{6\hat{p}},~\frac{1}{4}\left(\frac{1}{\hat{q}}-\frac{1}{\hat{p}}\right)\right\},~\widetilde{T}_1=\frac{24C_1}{M_1}+1,~C_2=\max\left\{\frac{1}{\beta}\ln\left(\frac{24M_2^L}{M_1}\right),~\widetilde{T}_1\right\},~
    \end{aligned}
\end{equation}
where $C_1$ is the constant introduced in \eqref{eq:R(k)} and $\beta$ is as defined in Case $1$. By Lemma~\ref{lem:Kronecker_approximation} and \eqref{eq:L layer dependence contour construction set up}, there exist $t_1>\widetilde{T}_1$ and an integer $p$ such that  
\begin{equation}\label{eq:L layer dependence subcase b Kronecker inequality}
\left|\tilde{\beta}_1t_1-2\hat{p}\alpha_1\pi-2p\hat{p}\pi\right|\leq 2\hat{p}\pi\varepsilon_1,\qquad
\left|\tilde{\beta}_{2^L}t_1-\frac{\hat{q}}{\hat{p}}\frac{\pi}{2}-2p\hat{q}\pi\right|\leq 2|\hat{q}|\pi\varepsilon_1.~
\end{equation}

Let $\widetilde{N}_1=\left[\frac{\tilde{\beta}_1}{\pi}\sqrt{t_1^2+C_2^2}-\frac{1}{2}\right]+1$. The circle $|k|=\frac{\widetilde{N}_1+\frac{1}{2}}{\tilde{\beta}_1}\pi$ meets the line $\Re k=-t_1$ at the points $\tilde{E}_1=\left(-t_1,-\sqrt{\left(\frac{\widetilde{N}_1+\frac{1}{2}}{\tilde{\beta}_1}\pi\right)^2-t_1^2}\right)$ and $\tilde{F}_1=\left(-t_1,\sqrt{\left(\frac{\widetilde{N}_1+\frac{1}{2}}{\tilde{\beta}_1}\pi\right)^2-t_1^2}\right)$; it meets the line $\Re k=t_1$ at the points $\tilde{E}_2=\left(t_1,-\sqrt{\left(\frac{\widetilde{N}_1+\frac{1}{2}}{\tilde{\beta}_1}\pi\right)^2-t_1^2}\right)$ and $\tilde{F}_2=\left(t_1,\sqrt{\left(\frac{\widetilde{N}_1+\frac{1}{2}}{\tilde{\beta}_1}\pi\right)^2-t_1^2}\right)$. The simple closed curve $\widetilde{\gamma}_1$ is then formed by the circular arcs $\overset{\frown}{\widetilde{E}_1\widetilde{E}_2}$ and $\overset{\frown}{\widetilde{F}_1\widetilde{F}_2}$, together with the line segments $\overline{\widetilde{E}_1\widetilde{F}_1}$ and $\overline{\widetilde{E}_2\widetilde{F}_2}$.

Step $2$. On the circular arcs $\overset{\frown}{\widetilde{E}_1\widetilde{E}_2}$ and $\overset{\frown}{\widetilde{F}_1\widetilde{F}_2}$, following the same argument as in Case 1, we obtain for $|d(k)|$ an estimate identical to \eqref{eq:tentative estimation of d(k) on arc independent}. We next derive a lower bound for $|d(k)|$ on the line segments $\overline{\widetilde{E}_1\widetilde{F}_1}$ and $\overline{\widetilde{E}_2\widetilde{F}_2}$.

From \eqref{eq:L layer dependence subcase b Kronecker inequality} we obtain
\begin{equation}\label{eq:L layer dependence subcase b beta1 and 2L sin bound}
\sin\left(\tilde{\beta}_{2^L}t_1\right)>0,\quad
\frac{1}{2}<\sin\left(\tilde{\beta}_1t_1\right)<1. 
\end{equation}
Combining \eqref{eq:L layer dependence subcase b first set up} and \eqref{eq:L layer dependence subcase b beta1 and 2L sin bound} yields
\begin{equation}\label{eq:L layer dependence subcase b coefficient sin same sign}
    \begin{aligned}
     \left(\tilde{A}_1\sin\left(\tilde{\beta}_1t_1\right)\right)\cdot\left(\tilde{A}_{2^L}\sin\left(\tilde{\beta}_{2^L}t_1\right)\right)>0.~
    \end{aligned}
\end{equation}
Using \eqref{eq:tilde_A_1}, \eqref{eq:upper and lower estimation for sin} and \eqref{eq:L layer dependence subcase b beta1 and 2L sin bound}, it follows immediately that
\begin{eqnarray}\label{eq:L layer dependence subcase b sin beta1 estimation}
\left|\Re\left(\tilde{A}_1\sin\left(\tilde{\beta}_1k\right)\right)\right|& \geq & \frac{1}{2}\left|\tilde{A}_1\right| \left|\sin\left(\tilde{\beta}_1\Re k\right)\right|\exp{\left(\tilde{\beta}_1|\Im k|\right)}  \nonumber\\
     &=&\frac{1}{2}\left|\tilde{A}_1\right| \left|\sin\left(\tilde{\beta}_1t_1\right)\right|\exp{\left(\tilde{\beta}_1|\Im k|\right)} \nonumber\\
     &\geq &\frac{2^L M_1}{8}\, \exp{\left(\tilde{\beta}_1|\Im k|\right)}.~
\end{eqnarray}
Based on \eqref{eq:upper bound for tilde_A_j2}, \eqref{eq:upper and lower estimation for sin}, and the jump condition for $n$ stated in the theorem, we obtain the estimate
  \begin{eqnarray}\label{eq:L layer dependence estimation for part two of main term on the line}
     \left| \Re\left(\sum_{j=2}^{2^L-1}\tilde{A}_j\sin\left(\tilde{\beta}_j k\right)\right) \right|&\leq &\sum_{j=2}^{2^L-1}\left|\tilde{A}_j\right|\exp{\left(|\tilde{\beta}_j| |\Im k|\right)}\leq \sum_{j=2}^{2^L-1}M_2^{L-1}\frac{\varepsilon_0}{2n_{*}}\exp{\left(|\tilde{\beta}_j| |\Im k|\right)} \nonumber\\
     &\leq &2^L M_2^{L-1}\frac{\varepsilon_0}{2n_{*}}\exp{\left(|\tilde{\beta}_j| |\Im k|\right)}\leq \frac{M_1 2^L}{24}\exp{\left(\tilde{\beta}_1 |\Im k|\right)}.~
  \end{eqnarray}

Following the same derivation as in case $1$, we obtain on $\widetilde{\gamma}_1$ an estimate for $|R(d(k))|$ that coincides with inequality \eqref{eq:arc independent parts of first term kill R(d(k))}.
Putting together \eqref{eq:arc independent parts of first term kill R(d(k))}, \eqref{eq:full_main_term_dk}, \eqref{eq:L layer dependence subcase b coefficient sin same sign}, \eqref{eq:L layer dependence subcase b sin beta1 estimation}, and \eqref{eq:L layer dependence estimation for part two of main term on the line}, we immediately obtain
\begin{eqnarray}\label{eq:L layer dependence subcase b estimation of dk on line}
      |d(k)|&=&\frac{1}{|k|}\left|kd(k)\right|\geq \frac{1}{|k|}\left| \left|\Re\left(kD(d(k))\right)\right|-\left|kR(d(k))\right| \right|  \nonumber\\
      &=&\frac{1}{|k|}\left| \left|\Re\left(\sum_{j=1}^{2^L}\frac{\tilde{A}_j}{2^L}\sin\left(\tilde{\beta}_jk\right)\right)\right|-\left|kR(d(k))\right| \right|  \nonumber \\
      &\geq& \frac{1}{|k|}\left| \left|\Re\left(\frac{\tilde{A}_1}{2^L}\sin\left(\tilde{\beta}_1k\right)+\frac{\tilde{A}_{2^L}}{2^L}\sin\left(\tilde{\beta}_{2^L}k\right)\right)\right|-\left|\Re\left(\sum_{j=2}^{2^L-1}\frac{\tilde{A}_j}{2^L}\sin\left(\tilde{\beta}_jk\right)\right)\right|-|kR(d(k))| \right| \nonumber \\
      &\geq& \frac{1}{|k|}\left| \left|\Re\left(\frac{\tilde{A}_1}{2^L}\sin\left(\tilde{\beta}_1k\right)\right)\right|-\frac{M_1}{24}\exp{\left(\tilde{\beta}_1|\Im k|\right)}-\frac{M_1}{24}\exp{\left(\tilde{\beta}_1|\Im k|\right)} \right| \nonumber \\
      &\geq& \frac{1}{|k|}\left|\frac{M_1}{8}\exp{\left(\tilde{\beta}_1|\Im k|\right)}-\frac{M_1}{24}\exp{\left(\tilde{\beta}_1|\Im k|\right)}-\frac{M_1}{24}\exp{\left(\tilde{\beta}_1|\Im k|\right)}  \right| \nonumber\\
      &=&\frac{M_1}{24|k|}\exp{\left(\tilde{\beta}_1|\Im k|\right)},~
    \end{eqnarray}
which holds on the line segments $\overline{\tilde{E}_1\tilde{F}_1}$ and $\overline{\tilde{E}_2\tilde{F}_2}$.~

Now, from \eqref{eq:estimation_f(k)},~\eqref{eq:tentative estimation of d(k) on arc independent} and \eqref{eq:L layer dependence subcase b estimation of dk on line},~we obtain
  \begin{equation*}
     |g(k)|\leq \frac{C_3\exp{\left( \left(\tilde{\beta}_1-\varepsilon\right)|\Im k| \right)}}{\frac{M_1}{24|k|}\exp{\left(\tilde{\beta}_1|\Im k|\right)}}=\frac{24C_3|k|}{M_1\exp{\left(\varepsilon |\Im k|\right)}}.~
  \end{equation*}
The rest of the proof proceeds similarly to Steps~3 through~5 in Case~1, which completes the proof.
\end{proof}


\section{Proof of theorem \ref{thm:uniqueness_C1}}

To prove Theorem \ref{thm:uniqueness_C1}, we first prove the following lemmas.

\begin{lemma}\label{lem:C2 approximate C1}
	A function $f \in C[0,1]$ is uniformly approximable by $C^{2}[0,1]$ functions with uniformly bounded first and second derivatives if and only if $f\in C^{1,1}[0,1]$.
\end{lemma}

\begin{proof} Step 1. Suppose that $f \in C^{1,1}[0,1]$. By definition, this means $f'$ is Lipschitz continuous, i.e., there exists a constant $L > 0$ such that
\begin{equation*}
	|f'(x_1) - f'(x_2)| \leq L |x_1 - x_2|, \quad \text{for all } x_1, x_2 \in [0,1].
\end{equation*}
In particular, this implies that $f'$ is uniformly continuous.

Let $\tilde{f}'$ be an extension of $f'$ on $\mathbb{R}$ as
\begin{equation*}
	\tilde{f}'(x) = 	
	\begin{cases}
		 f'(0),  & x< 0,             \\
		 f'(x),           & 0 \leq x \leq 1, \\
		 f'(1),  & x>1,
	\end{cases}
\end{equation*}
and define a sequence of functions $\{g_j\}$ by
\begin{equation}\label{eq:modified sequence}
	g_j(x) = (\tilde{f}' \ast \eta_{\frac{1}{j}})(x),
\end{equation}
where $\eta_{\frac{1}{j}}(x) = j\eta(jx)$ is a standard mollifier, and $\eta$ is a smooth bump function given by
\begin{equation*}
	\eta(x) =
	\begin{cases}
		C \exp\left(\dfrac{1}{x^2 - 1}\right), & |x| < 1, \\
		0, & |x| \geq 1.
	\end{cases}
\end{equation*}
Here, $C > 0$ is a normalization constant chosen so that $\int_{\mathbb{R}} \eta(x) \, dx = 1$.

Since $f'$ is uniformly continuous on $[0,1]$, the extension $\tilde{f}'$ defined in \eqref{eq:modified sequence} is uniformly continuous on $\mathbb{R}$. By the properties of the mollification (see \cite{Eva10}), the sequence $g_j$ converges uniformly to $\tilde{f}'$ on $\mathbb{R}$. This implies that the sequence $\{g_j\}$ is uniformly bounded; that is, there exists a constant $M$ such that
\begin{equation}\label{eq:bound for g_j}
	|g_j(x)| < M, \quad \text{for all } x \in \mathbb{R} \text{ and all } j.
\end{equation}
Moreover, since $\tilde{f}' = f'$ on $[0,1]$, it follows that $g_j$ converges uniformly to $f'$ on $[0,1]$. Furthermore, we have
\begin{eqnarray}\label{eq:g_j Lipschitz-continuous}
		\left|g_j(x_1) - g_j(x_2)\right|
		&=& \left|\int_{\mathbb{R}} \left(\tilde{f'}(x_1 - z) - \tilde{f'}(x_2 - z)\right) \eta_{\frac{1}{j}}(z) \mathrm{d} z\right| \nonumber\\
		&\leq& \int_{\mathbb{R}} \left|\tilde{f'}(x_1 - z) - \tilde{f'}(x_2 - z)\right| \eta_{\frac{1}{j}}(z) \mathrm{d} z \nonumber\\
		&\leq & L \left|x_1 - x_2\right| \text{.}
\end{eqnarray}
This implies that $g_j(x)$ is Lipschitz-continuous with a Lipschitz constant $L$ independent of $j$. Furthermore, from \eqref{eq:g_j Lipschitz-continuous}, it is evident that
\begin{equation}\label{eq:bound of derivative of g_j}
	|g'_j(x)| \leq L.
\end{equation}

Define $f_j(x) = \int_0^x g_j(t)  dt + f(0)$. Then we have
\begin{equation*}
	|f(x) - f_j(x)| = \left| \int_0^x [f'(t) - g_j(t)]  dt \right|.
\end{equation*}
Since $g_j$ converges uniformly to $f'$ on $[0,1]$, it follows that $f_j$ converges uniformly to $f$ on $[0,1]$. Moreover, by \eqref{eq:bound for g_j} and \eqref{eq:bound of derivative of g_j}, both the first and second derivatives of $f_j$ are uniformly bounded.

Step 2. Assume that $\{f_j\}$ is a sequence in $C^2[0,1]$ converging uniformly to $f$, and that there exists a constant $M > 0$ such that
\begin{equation}\label{eq:bound_of_derivatives}
	|f_j'(x)| < M \quad \text{and} \quad |f_j''(x)| < M, \quad \text{for all } x \in [0,1] \text{ and } j \in \mathbb{N}.
\end{equation}
The above bounds demonstrate that the sequence $\{f_j'\}$ is uniformly bounded and equicontinuous. By the Arzelà--Ascoli theorem (see \cite{Lax14}), there exists a subsequence $\{f_{n_j}'\}$ that converges uniformly to a function $g$. Since each $f_{n_j}'$ is continuous and uniformly bounded, the limit function $g$ is continuous as well. Thus, we have:
\begin{equation}\label{eq:approximation_derivative_approximation}	
		f_{n_j}' \rightarrow g~(\text{uniformly}),\quad
		f_{n_j} \rightarrow f~(\text{uniformly}), \quad
		f_{n_j}(x) = \displaystyle\int_0^x f_{n_j}'(t)  dt + f_{n_j}(0).
\end{equation}
Taking the limit as $j \to \infty$ in the third equation yields
\begin{equation*}
f(x) = \int_0^x g(t)  dt + f(0),
\end{equation*}
which shows that $f$ is differentiable with $f' = g$. Since $g \in C[0,1]$, $f \in C^{1}[0,1]$.
	
On the other hand, using \eqref{eq:bound_of_derivatives} and \eqref{eq:approximation_derivative_approximation}, we have
\begin{equation*}
\begin{aligned}
	\left| f'(x_1) - f'(x_2) \right| & \leq \left| f'_{n_j}(x_1) - f'(x_1) \right| + \left| f'_{n_j}(x_1) - f'_{n_j}(x_2) \right| + \left| f'_{n_j}(x_2) - f'(x_2) \right| \\
	& \leq \left| f'_{n_j}(x_1) - f'(x_1) \right| + M \left| x_1 - x_2 \right| + \left| f'_{n_j}(x_2) - f'(x_2) \right|,
\end{aligned}
\end{equation*}
which shows that $f'$ is Lipschitz continuous.
\end{proof}


\begin{lemma}\label{lem:estimation solution C1}
	Let $n\in C^{1,1}[0,1]$, the solution $ y $ of the equations \eqref{eq:original_ode} satisfies the following estimates:
	\begin{equation}\label{eq:y estimation C1}
		\left\{
		\begin{aligned}
			&\left| y(r) - \frac{1}{\left[n(0)n(r)\right]^{\frac{1}{4}} k} \sin\left(k \int_{0}^{r} \sqrt{n(t)} \mathrm{d} t\right) \right| \leq \frac{C}{|k|^2} \exp\left(\left|\Im k\right| \int_{0}^{r} \sqrt{n(t)} \mathrm{d} t\right), \text{~}\\
			&\left| y'(r) - \left[\frac{n(r)}{n(0)}\right]^{\frac{1}{4}} \cos\left(k \int_{0}^{r} \sqrt{n(t)} \mathrm{d} t\right) \right| \leq \frac{C}{|k|} \exp\left(\left|\Im k\right| \int_{0}^{r} \sqrt{n(t)} \mathrm{d} t\right). \text{~}\\
		\end{aligned}
		\right.
	\end{equation}
\end{lemma}

\begin{proof}
If $n \in C^{1,1}[0,1]$, then Lemma~\ref{lem:C2 approximate C1} guarantees the existence of an approximating sequence $\{n_j\} \subset C^2[0,1]$ that converges uniformly to $n$ and has uniformly bounded first and second derivatives. The estimates required by this lemma for each $n_j$ follow directly from \cite{AGP11,PT87}. The uniform boundedness of these derivatives completes the proof.
\end{proof}


By substituting the estimates for $ y $ and $ y' $ from \eqref{eq:y estimation C1} into the expression for $ d(k) $ given in \eqref{eq:d(k)}, we can immediately estimate the characteristic function $ d(k) $.

\begin{lemma}\label{lem:d(k) estimation C1}
	If $ n\in C^{1,1}[0,1] $, the following estimates hold:
\begin{equation}\label{eq:inequalites d(k) C1}
D(d(k))=\tilde{A}_1\frac{\sin{((1+\delta)k)}}{k}+\tilde{A}_2\frac{\sin{((1-\delta)k)}}{k}, \quad
|R(d(k))|<\frac{C_1}{|k|^2} \exp{\left\{(1+\delta)\left|\Im k\right|\right\}},
\end{equation}
where
\begin{equation*}
\tilde{A}_1=\frac{1}{2\left[n(0)\right]^{\frac{1}{4}}}\left(\left[n(1)\right]^{\frac{1}{4}}-\frac{1}{\left[n(1)\right]^{\frac{1}{4}}}\right), \quad \tilde{A}_2=\frac{1}{2\left[n(0)\right]^{\frac{1}{4}}}\left(\left[n(1)\right]^{\frac{1}{4}}+\frac{1}{\left[n(1)\right]^{\frac{1}{4}}}\right).
\end{equation*}
\end{lemma}
~\\


Now, we are at the stage to prove Theorem \ref{thm:uniqueness_C1}.~

\begin{proof}
Let $\tau$ denote the density of special transmission eigenvalues, which corresponds to the density of zeros of $d(k)$. By the Cartwright-Levinson Theorem (\cite{CCH16}), the density of zeros of $d(k)$ is given by $\frac{2(1+\delta)}{\pi}$. Thus, we have the relation
\begin{equation*}
	\frac{2(1+\delta)}{\pi} = \tau,
\end{equation*}
which means that given the density $\tau$ (or $\delta$), we need to determine the piecewise function $n$.	
	
Next, it follows from Lemma~\ref{lem:estimation solution C1} and the bounds $0 < n_{*} < n(r) < n^{*}$ that for $n(r) \in C^{1,1}[0,1]$, we obtain
\begin{eqnarray}\label{eq:y(r) final estimation inequality C1}
		\left|y(r)\right| & \leq & C\left\{\frac{1}{|k|}+\frac{1}{|k|^2} \right\}\exp\left(\left|\Im k\right|\int_{0}^{r}\sqrt{n(t)}\mathrm{d} t\right),
\end{eqnarray}
where $C$ is a generic constant. Then, using \eqref{eq:y(r) final estimation inequality C1} and \eqref{eq:def f g}, we derive
\begin{eqnarray*}
		\left|f(k)\right|& \leq & |k|^2C^2\left\{\frac{1}{|k|}+\frac{1}{|k|^2} \right\}^2\exp\left(2\left|\Im k\right|\int_{0}^{r}\sqrt{n(t)}\mathrm{d} t\right)\int_0^1\left|n^{(2)}(r)-n^{(1)}(r)\right|\mathrm{d}r \nonumber\\
		&\leq & C\exp\left\{(1+\delta_L-\varepsilon)\left|\Im k\right|\right\}.
\end{eqnarray*}
	
Furthermore, for $D(d(k))$, we have (\cite{MP94}):
\begin{equation*}
	\left|\sin((1+\delta)k) \right| > \frac{1}{4}\exp\left\{\left|\Im k\right|\left(1+\delta\right) \right\}, \quad \text{when } |k| = \frac{\left(N+\frac{1}{2}\right)\pi}{1+\delta}, \, N \in \mathbb{Z},
\end{equation*}
and
\begin{equation}\label{eq:estimation of A_1 and A_2}
	\begin{aligned}
		\left|\tilde{A}_1\right|>M_1,\quad	\left|\tilde{A}_2\right|<{M}_2,\\
	\end{aligned}
\end{equation}
where $M_1=\frac{1}{2{n^{*}}^{\frac{1}{4}}}\left|{[n(1)]}^{\frac{1}{4}}-\frac{1}{{[n(1)]}^{\frac{1}{4}}}\right|$ and $M_2=\frac{1}{2{n_{*}}^{\frac{1}{4}}}\max{\left\{{n_{*}}^{\frac{1}{4}}+\frac{1}{{n_{*}}^{\frac{1}{4}}},{n^{*}}^{\frac{1}{4}}+\frac{1}{{n^{*}}^{\frac{1}{4}}} \right\} }$.


Case $1$. $(1+\delta)$ and $(1-\delta)$ are rational linearly independent.\\
Recall the constant $T_1$ in Theorem\,\ref{thm:uniqueness} and let
\begin{equation}\label{eq:C1 independent set up notations}
		\alpha_1=\frac{1}{4}, \quad \alpha_2=0,~\varepsilon=\min{\left\{\frac{1}{12},\frac{1}{2\pi}\arcsin{\frac{M_1}{12M_2}}\right\}}, \quad C_2=\frac{\ln{12}+\ln{\frac{M_2}{M_1}}}{2\delta_L}.
\end{equation}

Since $(1+\delta)$ and $(1-\delta)$ are rationally linearly independent, it immediately follows that $\frac{1+\delta}{2\pi}$ and $\frac{1-\delta}{2\pi}$ are also rationally linearly independent. From Lemma~\ref{lem:Kronecker_approximation}, there exist $t_1 > T_1$ and integers $p_1, p_2 \in \mathbb{Z}$ such that
\begin{eqnarray}\label{eq:C1 independent Kronecker angle estimation}
		\left|\left(1+\delta\right)t_1-2\pi p_1-2\pi\alpha_1\right|<2\pi\varepsilon, \quad
		 \left|\left(1-\delta\right)t_1-2\pi p_2-2\pi\alpha_2 \right|<2\pi\varepsilon.
\end{eqnarray}

Now choose
\begin{equation*}
		N_1=\left[\frac{1+\delta}{\pi}\sqrt{C_2^2+t_1^2}-\frac{1}{2}\right]+1.
\end{equation*}
Consider the circle $|k| = \frac{(N_1 + \frac{1}{2})\pi}{1+\delta}$. This circle intersects the vertical line $\Re k = -t_1$ at two points, which we denote (from bottom to top) as $E_1$ and $F_1$. Similarly, it intersects the vertical line $\Re k = t_1$ at points $E_2$ and $F_2$ (also from bottom to top). Then the simple closed curve $\gamma_1$ consists of circular arcs $\stackrel{\frown}{E_1E_2}$ and $\stackrel{\frown}{F_1F_2}$, and line segments $\overline{E_1F_1}$ and $\overline{E_2F_2}$.

It remains to estimate $d(k)$ on the line segments $\overline{E_1F_1}$ and $\overline{E_2F_2}$. From \eqref{eq:C1 independent set up notations} and \eqref{eq:C1 independent Kronecker angle estimation}, we directly obtain
\begin{equation*}
|\left(1+\delta\right)t_1 - (2p_1\pi +\frac{\pi}{2})| \leq \frac{\pi}{6}, \quad
|\left(1-\delta\right)t_1-2p_2\pi | \leq  2\pi \varepsilon.
\end{equation*}
From the above inequalities, we have
\begin{equation*}
	\left|\sin(\left(1-\delta\right)t_1)\right|\leq \frac{M_1}{12M_2},
\end{equation*}
which, together with \eqref{eq:estimation of A_1 and A_2}, implies
\begin{equation}\label{eq:estimation for second part of the dominant term of d(k) in case1}
		\left|\frac{|\tilde{A}_2|}{|k|}\exp{\left\{(1-\delta)\left|\Im k\right|\right\}}\sin{((1-\delta)t_1)}\right|\leq \frac{M_1}{12|k|}\exp{\left\{(1+\delta)\left|\Im k\right|\right\}}.
\end{equation}
Moreover, on the segments $\overline{E_1F_1}$ and $\overline{E_2F_2}$, it follows from \eqref{eq:C1 independent set up notations} and the choice of $t_1$ that $|\Re (k)| = t_1 \geq T_1 \geq \frac{12C_1}{M_1}$. Therefore, we have $|k| \geq |\operatorname{Re}(k)| = t_1 \geq \frac{12C_1}{M_1}$. Consequently, it follows from Lemma\,\ref{lem:d(k) estimation C1} that
\begin{equation}\label{eq:independent case estimation for remainder term}
|R(d(k))| \leq	\left|\frac{C_1}{|k|^2}\exp{\left\{(1+\delta)\left|\Im k\right|\right\}}\right|\leq \frac{M_1}{12|k|}\exp{\left\{(1+\delta)\left|\Im k\right|\right\}}.
\end{equation}
		
Then, using \eqref{eq:expression of d(k) on arc independent}, \eqref{eq:bounds for real part of sin tilde beta_j k}, \eqref{eq:estimation of A_1 and A_2}, (\ref{eq:estimation for second part of the dominant term of d(k) in case1}), (\ref{eq:independent case estimation for remainder term}), and the triangle inequality, we find that on the line segments $\overline{E_1F_1}$ and $\overline{E_2F_2}$,
\begin{eqnarray*}
		|d(k)|	&\geq &\left|\frac{|\tilde{A}_1|}{2|k|}\exp{\left\{(1+\delta)\left|\Im k\right|\right\}}\sin{\left[(1+\delta)\Re{k}\right]}-\frac{|\tilde{A}_2|}{|k|}\exp{\left\{(1-\delta)\left|\Im k\right|\right\}}\sin{\left[(1-\delta)\Re{k}\right]}\right. \nonumber \\
		&&-\left.|R(d(k))| \right| \nonumber\\
		&\geq &\left|\frac{M_1}{4|k|}\exp{\left\{(1+\delta)\left|\Im k\right|\right\}}-\frac{M_2}{|k|}\exp{\left\{(1+\delta)\left|\Im k\right|\right\}}\frac{M_1}{12M_2}-\frac{C_1}{|k|^2}\exp{\left\{(1+\delta)\left|\Im k\right|\right\}}\right| \nonumber\\
		&\geq &\frac{M_1}{12|k|}\exp{\left\{(1+\delta)\left|\Im k\right|\right\}}.
\end{eqnarray*}
Therefore, $d(k)$ has a positive lower bound on the line segments $\overline{E_1F_1}$ and $\overline{E_2F_2}$.\\
	

Case $2$. $(1+\delta)$ and $(1-\delta)$ are rational linearly dependent.~

We only consider $n(1) > 1$ and $\delta > 1$. By Lemma~\ref{lem:d(k) estimation C1}, we have 
\begin{equation}\label{eq:A1A2>0}
\tilde{A}_1 > 0, \quad \tilde{A}_2 > 0. 
\end{equation}
Therefore, there exist $\hat{p}, \hat{q} \in \mathbb{Z}$ with $\hat{q} < 0$ and $0 < |\hat{q}| < \hat{p}$ such that
\begin{equation}\label{eq:C1 delta_m<1 linear relations}
	\begin{aligned}
		1-\delta=\frac{\hat{q}}{\hat{p}}\left(1+\delta\right).
	\end{aligned}
\end{equation}
Choose
\begin{eqnarray}\label{eq:C1 dependent case delta<1 n(1)<1 set up}
		\alpha_1&=&1+\frac{1}{\hat{p}+|\hat{q}|},~ \varepsilon=\min\left\{\frac{1}{2(\hat{p}+|\hat{q}|)},\frac{4|\hat{q}|(\hat{p}-|\hat{q}|)}{\hat{p}+|\hat{q}|}, \frac{4\hat{p}(\hat{p}-|\hat{q}|)}{\hat{p}+|\hat{q}|},\frac{|\hat{q}|}{2\hat{p}(\hat{p}+|\hat{q})|}\right\}, \nonumber\\
		\widetilde{M}_1&=&\min\left\{\left|\sin\left(\frac{\hat{p}-|\hat{q}|}{\hat{p}+|\hat{q}|}\pi+2\hat{p}\varepsilon\right)\right|,~ \left|\sin\left(\frac{\hat{p}-|\hat{q}|}{\hat{p}+|\hat{q}|}\pi-2\hat{p}\varepsilon\right)\right|\right\},\\
		\widetilde{T}_1&=&\max\left\{\frac{12C_1}{M_1}+1,\frac{4C_1}{M_1\widetilde{M}_1}+1\right\},\nonumber
\end{eqnarray}
where $C_1$ is defined in \eqref{eq:R(k)}.

From Lemma~\ref{lem:Kronecker_approximation} and \eqref{eq:C1 delta_m<1 linear relations},~there are  it follows immediately that there exist $\tilde{t}_1 > \widetilde{T}_1$ and $\tilde{p} \in \mathbb{Z}$ such that
\begin{eqnarray*}
	\left|\left(1+\delta\right)\tilde{t}_1-2\hat{p}\tilde{p}\pi-2\hat{p}\pi\alpha_1\right| \leq 2\hat{p}\pi\varepsilon, \quad  \left| \left(1-\delta\right)\tilde{t}_1-2\hat{q}\tilde{p}\pi-2\hat{q}\pi\alpha_1 \right| \leq 2|\hat{q}|\pi\varepsilon.
\end{eqnarray*}

Denote by $\widetilde{N}_1=\left[\frac{1+\delta}{\pi}\sqrt{\tilde{t}_1^2+C_2^2}-\frac{1}{2}\right]+1$, where $C_2$ is given by (\ref{eq:C1 independent set up notations}). The circle $|k|=\frac{\widetilde{N}_1+\frac{1}{2}}{1+\delta}\pi$ intersects the line segment $\Re k=-\tilde{t}_1$ at points $\widetilde{E}_1=\left(-\tilde{t}_1,-\sqrt{\left(\frac{\widetilde{N}_1+\frac{1}{2}}{1+\delta}\pi\right)^2-\tilde{t}_1^2}\right)$ and $\widetilde{F}_1=\left(-\tilde{t}_1,\sqrt{\left(\frac{\widetilde{N}_1+\frac{1}{2}}{1+\delta}\pi\right)^2-\tilde{t}_1^2}\right)$, and intersects the line segment $\Re k=\tilde{t}_1$ at points $\widetilde{E}_2=\left(\tilde{t}_1,-\sqrt{\left(\frac{\widetilde{N}_1+\frac{1}{2}}{1+\delta}\pi\right)^2-\tilde{t}_1^2}\right)$ and $\widetilde{F}_2=\left(\tilde{t}_1,\sqrt{\left(\frac{\widetilde{N}_1+\frac{1}{2}}{1+\delta}\pi\right)^2-\tilde{t}_1^2}\right)$. Then the simple closed curve $\tilde{\gamma}_1$ consists of circular arcs $\stackrel{\frown}{\widetilde{E}_1\widetilde{E}_2}$ and $\stackrel{\frown}{\widetilde{F}_1\widetilde{F}_2}$, and line segments $\overline{\widetilde{E}_1\widetilde{F}_1}$ and $\overline{\widetilde{E}_2\widetilde{F}_2}$.
	
It remains to estimate $d(k)$ on the line segments $\overline{\widetilde{E}_1\widetilde{F}_1}$ and $\overline{\widetilde{E}_2\widetilde{F}_2}$.
From the definition of $\alpha_1$ and $\varepsilon$ in \eqref{eq:C1 dependent case delta<1 n(1)<1 set up}, we have
\begin{equation*}
\left|\left(1+\delta\right)\tilde{t}_1-\left(2\hat{p}\tilde{p}\pi+\pi+\frac{\hat{p}-|\hat{q}|}{\hat{p}+|\hat{q}|}\pi\right)\right|\leq 2\hat{p}\pi\varepsilon, \quad
\left|\left(1-\delta\right)\tilde{t}_1-\left(-2|\hat{q}|\tilde{p}\pi-\pi+\frac{\hat{p}-|\hat{q}|}{\hat{p}+|\hat{q}|}\pi\right)\right|\leq 2|\hat{q}|\pi\varepsilon,
\end{equation*}
which, together with the choice of $\varepsilon$, implies that
\begin{equation}\label{eq:estimation of 1delta}
\sin((1+\delta)\tilde{t}_1) \cdot \sin((1-\delta)\tilde{t}_1) > 0.
\end{equation}

Moreover, from \eqref{eq:C1 dependent case delta<1 n(1)<1 set up},~it is evident that
	\begin{equation}\label{eq:estimation of sin 1+delta t_1 delta>1 n(1)>1}
		|\sin\left(\left(1+\delta\right)\tilde{t}_1\right)| \geq \widetilde{M}_1.
	\end{equation}
On the segments $\overline{\widetilde{E}_1\widetilde{F}_1}$ and $\overline{\widetilde{E}_2\widetilde{F}_2}$, we have $|k| \geq |\Re k|=\tilde{t}_1 \geq T_0 \geq \frac{4C_1}{M_1\widetilde{M}_1}$, then
\begin{equation}\label{eq:estimation of the remainder term of d(k) case 2b}
	\begin{aligned}
		\left|\frac{C_1}{|k|^2}\exp{\left\{(1+\delta)\left|\Im k\right|\right\}}\right| \leq \frac{M_1\widetilde{M}_1}{4|k|}\exp{\left\{(1+\delta)\left|\Im k\right|\right\}}.~
	\end{aligned}
\end{equation}
From Lemma\,\ref{lem:d(k) estimation C1}, \eqref{eq:bounds for real part of sin tilde beta_j k}, (\ref{eq:A1A2>0}), (\ref{eq:estimation of 1delta}), (\ref{eq:estimation of sin 1+delta t_1 delta>1 n(1)>1}), and \eqref{eq:estimation of the remainder term of d(k) case 2b}, we find that on the line segments  $\overline{\widetilde{E}_1\widetilde{F}_1}$ and $\overline{\widetilde{E}_2\widetilde{F}_2}$,
\begin{equation*}
	\begin{aligned}
		\left|d(k)\right|& \geq \left|\frac{\tilde{A}_1}{2|k|}\exp{\left\{(1+\delta)\left|\Im k\right|\right\} }\sin{((1+\delta)\tilde{t}_1)}+\frac{\tilde{A}_2}{2|k|}\exp{\left\{(1-\delta)\left|\Im k\right|\right\} }\sin{((1-\delta)\tilde{t}_1)}-|R(d(k))|\right|\\
		& \geq \left|\frac{M_1\tilde{M_1}}{2|k|}\exp{\left\{(1+\delta)\left|\Im k\right|\right\}}-\frac{C_1}{|k|^2}\exp{\left\{(1+\delta)\left|\Im k\right|\right\} }\right|\\
		& \geq \left|\frac{M_1\tilde{M_1}}{4|k|}\exp{\left\{(1+\delta)\left|\Im k\right|\right\} }\right|.~
	\end{aligned}
\end{equation*}
Therefore, $d(k)$ has a positive lower bound on the line segments $\overline{\widetilde{E}_1\widetilde{F}_1}$ and $\overline{\widetilde{E}_2\widetilde{F}_2}$.

The remaining arguments parallel those in proof of Theorem\,\ref{thm:uniqueness}, which completes the proof.

\end{proof}


\section{Conclusion and Future Work}\label{conclusion}

In this paper, we study the unique determination of the refractive index $n$ using a special transmission eigenvalue combined with the values of $n$ on a subinterval. We consider two regularity classes: $n \in C_{p}^{2}[0,1] \cap \mathcal{M}$ or $n \in C^{1,1}[0,1] \cap \mathcal{M}$.

Our future goals are twofold: first, to establish uniqueness with minimal additional information---for instance, using only the transmission eigenvalue---and second, to further relax the regularity assumptions on $n$, such as considering $n \in L^{\infty}[0,1]$.


\section*{Acknowledgment}
The work of Y. Jiang is supported in part by the National Key R\&D Program of China (2024YFA1012302) and China Natural National Science Foundation (No. 123B2017). The work of K. Zhang is supported in part by China Natural National Science Foundation (Grant No. 12271207).


\begin{thebibliography}{99}	

    \bibitem{AGP11}
    {T.Aktosun, D.~Gintides and V.G.~Papanicolaou},
    {\it The uniqueness in the inverse problem for transmission eigenvalues for the spherically symmetric variable-speed wave equation},
    Inverse Probl., \textbf{27(21)}, 2011, pp. 115004.

    \bibitem{AS15}
    {J.~An and J.~Shen},
    {\it Spectral approximation to a transmission eigenvalue problem and its applications to an inverse problem},
    Comput. Math. Appl., \textbf{69(10)}, 2015, pp. 1132--1143.

    \bibitem{CCG10}
    {F.~Cakoni, D.~Colton and D.~Gintides},
    {\it The interior transmission eigenvalue problem},
    SIAM J. Math. Anal., \textbf{42(6)}, 2010, pp. 2912--2921.

    \bibitem{CCH09}
    {F.~Cakoni, D.~Colton and H.~Haddar},
    {\it The computation of lower bounds for the norm of the index of refraction in an anisotropic media from far field data},
    J. Integral Equations Appl., \textbf{21(2)}, 2009, pp. 203--227.

    \bibitem{CCH10}
    {F.~Cakoni, D.~Colton and H.~Haddar},
    {\it On the determination of {D}irichlet or transmission eigenvalues from far field data},
    C. R. Math. Acad. Sci. Paris, \textbf{348(7-8)}, 2010, pp. 379--383.

    \bibitem{CCH16}
    {F.~Cakoni, D.~Colton and H.~Haddar},
    {\it Inverse Scattering Theory and Transmission Eigenvalues},
    SIAM-Society for Industrial and Applied Mathematics, Philadelphia, PA, USA, 2016.

    \bibitem{CCH12}
    {F.~Cakoni, A.~Cossonni\`ere and H.~Haddar},
    {\it Transmission eigenvalues for inhomogeneous media containing obstacles},
    Inverse Probl. Imaging, \textbf{6(3)}, 2012, pp. 373--398.

    \bibitem{CGH10}
    {F.~Cakoni, D.~Gintides and H.~Haddar},
    {\it The existence of an infinite discrete set of transmission eigenvalues},
    SIAM J. Math. Anal., \textbf{42(1)}, 2010, pp. 237--255.

    \bibitem{CH03}
    {F.~Cakoni and H.~Haddar},
    {\it Interior transmission problem for anisotropic media},
	Mathematical and Numerical Aspects of Wave Propagation WAVES 2003: Proceedings of The Sixth International Conference on Mathematical and Numerical Aspects of Wave Propagation Held at Jyv{\"a}skyl{\"a}, Finland, 30 June--4 July 2003, Springer, 2003, pp. 613--618.
	
	\bibitem{CH07}
	{F.~Cakoni and H.~Haddar},
	{\it A variational approach for the solution of the electromagnetic interior transmission problem for anisotropic media},
	Inverse Probl. Imaging, \textbf{1(3)}, 2007, pp. 443--456.
	
	\bibitem{CMS14}
	{F.~Cakoni, P.~Monk and J.G.~Sun},
	{\it Error analysis for the finite element approximation of transmission eigenvalues},
	Comput. Methods Appl. Math., \textbf{14(4)}, 2014, pp. 419--427.
	
	\bibitem{CMR15}
	{F.~Cakoni, S.~Moskow and S.~Rome},
	{\it The perturbation of transmission eigenvalues for inhomogeneous media in the presence of small penetrable inclusions},
	Inverse Probl. Imaging, \textbf{9(3)}, 2015, pp. 725--748.
	
	\bibitem{CKP89}
	{D.~Colton, A.~Kirsch and L.~ P\"aiv\"arinta},
	{\it Far-field patterns for acoustic waves in an inhomogeneous medium},
	SIAM J. Math. Anal., \textbf{20(6)}, 1989, pp. 1472--1483.
	
	\bibitem{CK19}
	{D.~Colton and R.~Kress},
	{\it Inverse acoustic and electromagnetic scattering theory},
	Fourth ed., Applied Mathematical Sciences 93, Springer Science \& Business Media, 2019.
	
	\bibitem{CL13}
	{D.~Colton and Y.J.~Leung},
	{\it Complex eigenvalues and the inverse spectral problem for transmission eigenvalues},
	Inverse Probl., \textbf{29(10)}, 2013, pp. 104008.
	
	\bibitem{CL17}
	{D.~Colton and Y.J.~Leung},
	{\it The existence of complex transmission eigenvalues for spherically stratified media},
	Appl. Anal., \textbf{96(1)}, 2017, pp. 39--47.
	
	\bibitem{CLM15}
    {D.~Colton, Y.J.~Leung and S.X. Meng},
    {\it Distribution of complex transmission eigenvalues for spherically stratified media},
    Inverse Probl., \textbf{31(3)}, 2015, pp. 035006.
	
	\bibitem{CM88}
	{D.~Colton and P.~Monk},
	{\it The inverse scattering problem for time-harmonic acoustic waves in an inhomogeneous medium},	
	Q. J. Mech. Appl. Math., \textbf{41(1)}, 1988, pp. 97--125.
	
	\bibitem{CPS07}
	{D.~Colton, L.~P\"aiv\"arinta and J.~Sylvester},
	{\it The interior transmission problem},
	Inverse Probl. Imaging, \textbf{1(1)}, 2007, pp. 12--28.
	
	\bibitem{Eva10}
	{L.C.~Evans},
	{\it Partial differential equations},
	Second ed., Graduate studies in mathematics 19, American Mathematical Society, 2010.
		
	\bibitem{GJSX16}
	{H.~Geng, X.~Ji, J.G.~Sun and L.~Xu},
	{\it $C^{0}$iP methods for the transmission eigenvalue problem},
	J. Sci. Comput., \textbf{68}, 2016, pp. 326--338.
	
	\bibitem{GP17}
	{D.~Gintides and N.~Pallikarakis},
	{\it The inverse transmission eigenvalue problem for a discontinuous refractive index}, Inverse Probl., \textbf{33}, 2017, pp. 055006.
	
	\bibitem{HW80}
	{G.H.~Hardy and E.M.~Wright},
	{\it An introduction to the theory of numbers},
	Fifth ed., Oxford University Press, 1980.
	
	\bibitem{LV15}
	{E.~Lakshtanov and B.~Vainberg},
	{\it Sharp Weyl law for signed counting function of positive interior transmission eigenvalues},
	SIAM J. Math. Anal., \textbf{47}, 2015, pp. 3212--3234.
	
	\bibitem{Lax14}
	{P.D.~Lax},
	{\it Functional analysis},
	John Wiley \& Sons, 2014.
	
	\bibitem{LC12}	
	{Y.J.~Leung and D.~Colton},
	{\it Complex transmission eigenvalues for spherically stratified media},
	Inverse Probl., \textbf{28}, 2012, pp. 075005.
	
	\bibitem{LHLL15}
	{T.~Li, W.Q.~Huang, W.W.~Lin and J.~Liu},
	{\it On spectral analysis and a novel algorithm for transmission eigenvalue problems},
	J. Sci. Comput., \textbf{64}, 2015, pp. 83--108.
	
	\bibitem{MP94}
	{J.R.~Mclaughlin and P.L.~Polyakov},
	{\it On the uniqueness of a spherically symmetric speed of sound from transmission
	eigenvalues},
	J. Differ. Equations, \textbf{107}, 1994, pp. 351--382.	
	
	\bibitem{MPS94}
	{J.R.~Mclaughlin, P.L.~Polyakov, and P.E.~Sacks},
	{\it Reconstruction of a spherically symmetric speed of sound},
	SIAM J. Appl. Math., \textbf{54}, 1994, pp. 1203--1223.
	
	\bibitem{PS08}
	{L.~P\"aiv\"arinta and J.~Sylvester},
	{\it Transmission eigenvalues},
	SIAM J. Math. Anal.,  \textbf{40}, 2008, pp. 738--753.
	
	\bibitem{PT87}
	{J.~P{\"o}schel and E.~Trubowitz},
	{\it Inverse spectral theory},
	Pure and Applied Mathematics 130, Academic Press, 1987.
	
	\bibitem{Ram09}
	{A.G.~Ramm},
	{\it Property C for ODE and applications to an inverse problem for a heat equation},
	Bull. Pol. Acad. Sci. Math., \textbf{57}, 2009, pp. 243--249.
	
	\bibitem{RS91}
	{B.P.~Rynne and B.D.~Sleeman},
	{\it The interior transmission problem and inverse scattering from inhomogeneous media}, SIAM J. Math. Anal., \textbf{22}, 1991, pp. 1755--1762.
	
	\bibitem{SS10}
	{E.M.~Stein and R.~Shakarchi},
	{\it Complex analysis},
	Princeton Lectures in Analysis 2, Princeton University Press, 2010.
		
		
\end{thebibliography}
\end{document}